\documentclass[reqno,11pt]{amsart}
\usepackage[utf8]{inputenc}
\usepackage{amsthm,amsfonts,amstext,amssymb,amsmath,latexsym,mathtools} 
\usepackage{graphics}
\usepackage{graphicx}
\usepackage[abs]{overpic} 
\usepackage{enumerate}

\usepackage{bm}
\usepackage{verbatim}
\usepackage{wrapfig}
\usepackage{tikz}
\theoremstyle{plain}
\newtheorem{theorem}{Theorem}
\newtheorem{corollary}[theorem]{Corollary}
\newtheorem{proposition}[theorem]{Proposition}
\newtheorem{lemma}[theorem]{Lemma}

\theoremstyle{definition}

\theoremstyle{remark}
 
\numberwithin{equation}{section} 

\numberwithin{equation}{section}

\newcommand{\bbR}{\mathbb{R}}
\newcommand{\bbC}{\mathbb{C}}

\newcommand{\bbD}{\mathbb{D}}
\newcommand{\bbN}{\mathbb{N}}
\newcommand{\bbT}{\mathbb{T}}

\newcommand{\mcg}{\mathcal{G}}
\newcommand{\mcgn}{\mathcal{G}^{(n)}}
\newcommand{\mcl}{\mathcal{L}}
\newcommand{\mcm}{\mathcal{M}}
\newcommand{\mcs}{\mathcal{S}}
\newcommand{\mct}{\mathcal{T}}

\DeclareMathOperator{\rank}{\textrm{rank}}

\title{On foci of ellipses inscribed in cyclic polygons}

\author[M.~Hunziker]{Markus Hunziker}

\address[MH]{Department of Mathematics, Baylor University, Waco TX, USA}

\email{Markus\_Hunziker@baylor.edu}

\author[A. Mart\'{\i}nez-Finkelshtein]{Andrei Mart\'{\i}nez-Finkelshtein}

\address[AMF]{Department of Mathematics, Baylor University, Waco TX, USA, and Department of Mathematics, University of Almer\'{\i}a, Almer\'{\i}a, Spain}

\email{A\_Martinez-Finkelshtein@baylor.edu}

\author[T.~Poe]{Taylor Poe}

\address[TP]{Department of Mathematics, Baylor University, Waco TX, USA}

\email{Taylor\_Thompson8@baylor.edu}

\author[B.~Simanek]{Brian Simanek}

\address[BS]{Department of Mathematics, Baylor University, Waco TX, USA}

\email{Brian\_Simanek@baylor.edu}

\keywords{Orthogonal polynomials, Poncelet Ellipses, Blaschke products }

\subjclass[2010]{Primary: 30J10, 42C05;  Secondary: 14N15, 14H50, 47A12}

\begin{document}

\begin{abstract} 
	Given a natural number $n\geq3$ and two points $a$ and $b$ in the unit disk $\bbD$ in the complex plane, it is known that there exists a unique elliptical disk having $a$ and $b$ as foci that can also be realized as the intersection of a collection of convex cyclic $n$-gons whose vertices fill the whole unit circle $\bbT$.  What is less clear is how to find a convenient formula or expression for such an elliptical disk.  Our main results reveal how orthogonal polynomials on the unit circle provide a useful tool for finding such a formula for some values of $n$.  The main idea is to realize the elliptical disk as the numerical range of a matrix and the problem reduces to finding the eigenvalues of that matrix.
\end{abstract}

\maketitle

\section{Introduction}\label{intro}

Suppose $n\geq3$ is a natural number and $E$ is an ellipse in the open unit disk $\bbD$ in the complex plane.  A classical result known as \textit{Poncelet's Theorem} asserts that if there is an $n$-gon $P$ inscribed in the unit circle $\bbT$ with every side of $P$ tangent to $E$, then there are in fact infinitely many such $n$-gons and the union of the vertices of these $n$-gons fills $\bbT$.  In this case, the ellipse $E$ is said to be a \textit{Poncelet $n$-ellipse}.  A simple argument shows that if $a,b\in\bbD$ and $n\geq3$ is a natural number, then there exists a unique Poncelet $n$-ellipse with foci at $a$ and $b$.  However, if $n>3$, then it is not obvious how to write down a formula for this ellipse or deduce any properties of its size (such as area, eccentricity, etc.).  Some early relevant formulas for this purpose were found by Cayley \cite[Chapter 5]{DrRa}, but they are not easy formulas to use.  Some of the main results in this paper will show how to write an explicit expression for Poncelet $n$-ellipses when $n=4$ or $n=6$.

\begin{figure}[htb]
	\begin{center}
	\includegraphics[width=0.35\linewidth]{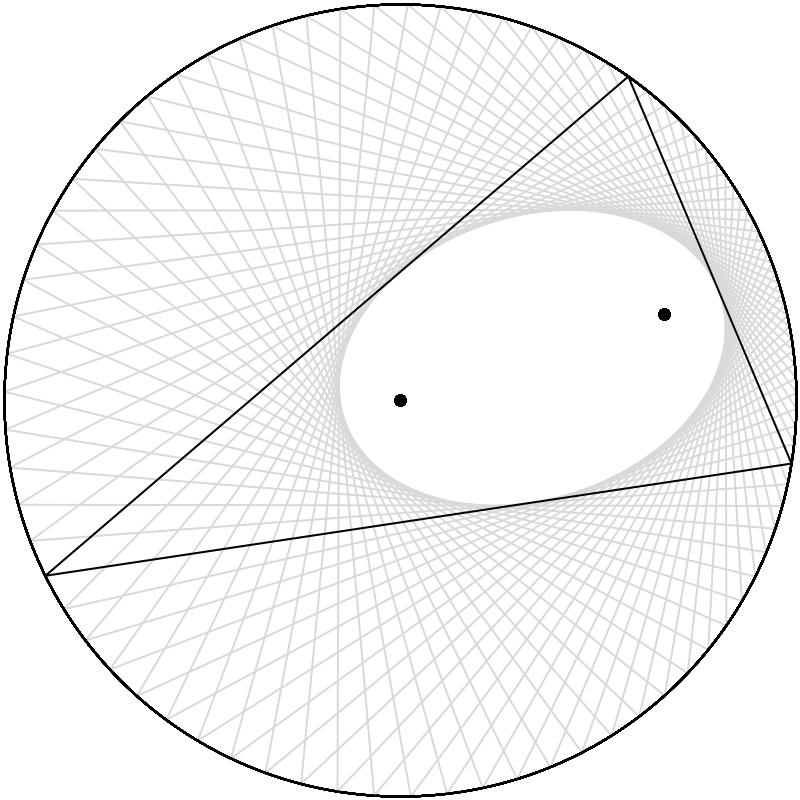}
	\caption{A Poncelet 3-ellipse and its two foci.}
	\end{center}
	\label{fig:PonceletEllpse}
\end{figure}

To accomplish this task, we must frame this problem in a broader context.  The main tool we will use is numerical ranges of a special class of $(n-1)\times (n-1)$ matrices called completely non-unitary contractions of defect index $1$ (denoted $S_{n-1}$).  The specifics of these objects are provided in Section \ref{sec:background} below, but for now it is enough for us to say that the numerical range $W(A)$ of a matrix $A\in S_{n-1}$ is a strictly convex subset of the unit disk with a smooth boundary (see \cite{MFSS}).  Furthermore, the boundary of this set $\partial W(A)$ has the Poncelet property, meaning that every point on $\bbT$ is the vertex of an $n$-gon that is inscribed in $\bbT$ having every side tangent to $\partial W(A)$ and every point on $\partial W(A)$ is a point of tangency for such an $n$-gon (see \cite{DGV}).  Our approach to the problem described in the previous paragraph has its roots in work of Gau and Wu \cite{GW98} and aims to realize the desired ellipse as the numerical range of an appropriate matrix in $S_{n-1}$.  Theorem \ref{Shukla} below assures us that this problem has a solution.

One simplification of our problem comes from the fact that instead of finding a matrix in $S_{n-1}$ with the desired properties, it suffices to find just the eigenvalues of that matrix.  This is because numerical ranges are preserved by unitary conjugation and every matrix in $S_{n-1}$ is unitarily equivalent to a canonical form called a cutoff CMV matrix (see \cite{HMFPS,MFSS}).    Such matrices are an important part of the theory of orthogonal polynomials on the unit circle and will play an essential role in our analysis.  Further details are presented in Section \ref{sec:background}.

Returning to our original problem, notice that Theorem \ref{Shukla} states that both $a$ and $b$ must be eigenvalues of the desired matrix.  Thus, one really only needs to determine the remaining $n-3$ eigenvalues, which is why the problem becomes trivial when $n=3$.  When $n=4$, a concise formula for the one remaining eigenvalue appears in \cite{GW}, and we give another proof of that formula in Section \ref{sec:quadrilateral}.  In \cite{Mirman1}, Mirman presented a collection of algebraic relationships that must be satisfied by the eigenvalues we seek.  While this finding is significant, it falls short of presenting a complete solution to our problem because the algebraic relationships admit multiple solutions (even in $\bbD^{n-3}$).  In Section \ref{sec:hexagon}, we will present a different collection of algebraic relationships that applies in the case $n=6$ and admits a unique solution in $\bbD^{3}$, which means that the unique solution to this system is the collection of $3$ eigenvalues that we seek.  In Section \ref{sec:pentagon}, we will examine some additional properties of all of the solutions to Mirman's system of equations in the case $n=5$.

The next section is a brief review of the background, notation, and terminology that will be relevant to the remainder of the paper.  Many of these topics are discussed in much greater detail in \cite{GBook,HMFPS}, and a thorough introduction to orthogonal polynomials on the unit circle can be found in \cite{OPUC1,OPUC2}.

% % % % % % % % % % % % % % % % % % % % % % % % % % % % % % % % % % % % % % % % % % % % % % % % % % % % % % % % % % % % % % 
\section{Background $\&$ Notation} \label{sec:background}
% % % % % % % % % % % % % % % % % % % % % % % % % % % % % % % % % % % % % % % % % % % % % % % % % % % % % % % % % % % % % % 

Our primary objects of study will be numerical ranges of matrices.  The \emph{numerical range} of  a matrix $A\in \bbC^{n\times n}$ is the subset of the complex plane $\bbC$ given by
$$
W(A)=\{ \langle x,Ax\rangle:\, x\in \bbC^n, \; \|x\|=1 \}.
$$
For any matrix $A$, the set $W(A)$ is a compact and convex subset of $\bbC$ (a fact known as the Toeplitz-Hausdorff Theorem) that contains the eigenvalues of $A$.  If $W(A)$ is bounded by an ellipse, then we will say that $W(A)$ is an \textit{elliptical disk}.

The matrices $A$ that we are most interested in have the following properties:
\begin{enumerate}
\item[\rm{(i)}] $\| A\|=1$;
\item[\rm{(ii)}] all eigenvalues of $A$ are in $\bbD$;
\item[\rm{(iii)}]
$
\rank(I-AA^*)=\rank(I-A^*A)=1.
$
\end{enumerate}
The set of all $n\times n$ matrices satisfying properties (i-iii) is precisely the set $S_n$ that we described in the Introduction.  As we stated there, it is known that the numerical range of a matrix in $S_n$ is the convex hull of an algebraic curve of class $n$ and has the $(n+1)$-Poncelet property, meaning that every point on $\partial W(A)$ is a point of tangency for a convex $(n+1)$-gon that is circumscribed about $W(A)$ and inscribed in $\bbT$ (see \cite{HMFPS} for details).  

There are several canonical forms of matrices from the class $S_n$ (see \cite[Section 2.3]{HMFPS}) and the one that we will use is that of a cutoff CMV matrix.  To define a CMV matrix, first define a sequence of $2\times2$ matrices $\{\Theta_j\}_{j=0}^{\infty}$ by
\[
\Theta_j=\begin{pmatrix}
\bar{\alpha}_j & \sqrt{1-|\alpha_j|^2}\\
\sqrt{1-|\alpha_j|^2} & -\alpha_j
\end{pmatrix},
\]
where $\alpha_j\in\bbD$.  One then defines the operators $\mcl$ and $\mcm$ by
\[
\mcl=\Theta_0\oplus\Theta_2\oplus\Theta_4\oplus\cdots,\qquad\mcm=1\oplus\Theta_1\oplus\Theta_3\oplus\cdots
\]
where the initial $1$ in the definition of $\mcm$ is a $1\times1$ identity matrix.  The \textit{CMV matrix} corresponding to the sequence $\{\alpha_n\}_{n=0}^{\infty}$ is
\begin{equation} \label{cmvInf}
\bm{\mathcal G} :=\mcl\mcm=\begin{pmatrix}
\overline{\alpha_0} & \overline{\alpha_1} \rho_0 & \rho_1 \rho_0  & 0 & 0 & \dots \\
\rho_0 &- \overline{\alpha_1} \alpha_0 & -\rho_1 \alpha_0  & 0 & 0 & \dots \\
0 & \overline{\alpha_2} \rho_1 & - \overline{\alpha_2} \alpha_1   &  \overline{\alpha_3} \rho_2  & \rho_3 \rho_2 & \dots \\
0 & \rho_2 \rho_1 & - \rho_2 \alpha_1   &  -\overline{\alpha_3} \alpha_2  &- \rho_3 \alpha_2 & \dots \\
0 & 0 & 0  &   \overline{\alpha_4} \rho_3  & - \overline{\alpha_4} \alpha_3 & \dots \\
\dots & \dots & \dots  &   \dots  & \dots & \dots
\end{pmatrix}, \quad \rho_n=\sqrt{1-|\alpha_n|^2}
\end{equation}
(see \cite[Section 4.2]{OPUC1}).  Since each of $\mcl$ and $\mcm$ is a direct sum of unitary matrices, each of $\mcl$ and $\mcm$ is unitary and hence $\mcg$ is unitary as an operator on $\ell^2(\bbN)$.  The principal $n\times n$ submatrix of $\mcg$ will also be called the $n\times n$ \textit{cut-off CMV matrix}, which we will denote by $\mcg^{(n)}$.  It is easy to see that $\mcg^{(n)}\in S_n$.

CMV matrices are intimately connected with the theory of orthogonal polynomials on the unit circle (OPUC).  Indeed, if one defines
\begin{equation}
\label{characteristicequ}
\Phi_{n}(z):=\det(zI_n-\mcg^{(n)}),
\end{equation}
then the polynomial $\Phi_n(z)$ is the degree $n$ monic orthogonal polynomial with respect to the measure $\mu$ that is the spectral measure of $\mcg$ and the vector $\vec{e}_1$.  
%If we need to make explicit reference to the measure $\mu$, then we will write $\Phi_{n,\mu}$.  
Since $\mcg$ is unitary, the formula \eqref{characteristicequ} implies that all zeros of $\Phi_n$ are in $\bbD$.  Furthermore, the coefficients $\{\alpha_n\}_{n=0}^{\infty}$ that are used to define $\mcg$ are related to $\{\Phi_n\}_{n=0}^{\infty}$ by the Szeg\H{o} recursion:
\begin{equation} \label{szego}
\begin{pmatrix}
\Phi_{k+1}(z)\\ \Phi_{k+1}^*(z)
\end{pmatrix}= \begin{pmatrix}
z & -\overline{\alpha_k} \\ -\alpha_k z & 1
\end{pmatrix} \,
\begin{pmatrix}
\Phi_{k}(z)\\ \Phi_{k}^*(z)
\end{pmatrix}  , 
\end{equation}
where if
\[
\Phi_n(z)=\sum_{j=0}^n c_j z^j,
\]
then
\begin{equation} \label{reversed}
\Phi^*_n(z)=\sum_{j=0}^n \overline{c_j} z^{n-j}=z^n \, \overline{\Phi_n\left( 1/\overline{z} \right)}.
\end{equation}
$\Phi_n^*$ is called the reversed polynomial of $\Phi_n $.  Observe that $\Phi^*_n$ can be of degree strictly less than $n$.  It follows from the Szeg\H{o} recursion that $\alpha_n=- \overline{\Phi_{n+1}(0)}$ and the sequence $\{\alpha_n\}_{n=0}^{\infty}$ is often called the sequence of \textit{Verblunsky coefficients} for the measure $\mu$.  For future use, let us define the notation
\[
\Phi_{n}(z)=\mcs_{\alpha_{n-1}}(\Phi_{n-1}(z)),\qquad \Phi_{n}^*(z)=\mct_{\alpha_{n-1}}(\Phi_{n-1}^*(z)).
\]
to say that $\Phi_n$ is related to $\Phi_{n-1}$ by the Szeg\H{o} recursion and the parameter $\alpha_{n-1}$.  

Some of the most important theorems in the study of OPUC come from establishing the following bijections (see \cite[Chapter 1]{OPUC1}):
\begin{enumerate}
\item[$\bullet$]  All of the zeros of $\Phi_{n,\mu}(z)$ are in $\bbD$ and any collection $\{z_j\}_{j=1}^n\in\bbD^n$ is the zero set of $\Phi_{n,\nu}(z)$ for some $\nu$ supported on $\bbT$.
\item[$\bullet$]    The sequence $\{\Phi_{n,\mu}(0)\}_{n\in\bbN}$ is a sequence in $\bbD$ and every sequence $\{\gamma_j\}_{j\in\bbN}\in\bbD^{\bbN}$ satisfies $\Phi_{n,\nu}(0)=\gamma_n$ for all $n\in\bbN$ and some measure $\nu$ supported on $\bbT$.
\end{enumerate}
This last fact (known as Verblunsky's Theorem \cite[Section 1.7]{OPUC1}) tells us that the sequence $\{\alpha_n\}_{n=0}^{\infty}$ completely characterizes the measure $\mu$ on $\bbT$.  
If it is necessary to make the sequence of Verblunsky coefficients explicit, we will write $\Phi_n(z;\alpha_0,\ldots,\alpha_{n-1})$.  We also note here that the Szeg\H{o} recursion is invertible.  This means that if we know $\Phi_n(z;\alpha_0,\ldots,\alpha_{n-1})$, then we can recover $\Phi_j(z;\alpha_0,\ldots,\alpha_{j-1})$ for all $j<n$ and hence we can recover $\alpha_j$ for all $j=0,\ldots,n-1$ (by evaluation at $0$).

Given a family of OPUC, replacing the last Verblunsky coefficient $\alpha_{n-1}$ in the Szeg\H{o} recursion with $\lambda \in \bbT$ yields a degree-$n$ \textit{paraorthogonal polynomial} on the unit circle (POPUC)
$$
\Phi_n(z;\alpha_0,...,\alpha_{n-2},\lambda)=z \Phi_{n-1}(z;\alpha_0,...,\alpha_{n-2})-\overline{\lambda} \Phi^*_{n-1}(z;\alpha_0,...,\alpha_{n-2}).
$$
In contrast to OPUC, the zeros of POPUC are on $\bbT$. Given $\Phi_{n-1}(z)$, we can define $\{\Phi_n^{(\lambda)}\}$ as the set of all degree-$n$ POPUC for $\Phi_{n-1}$ as $\lambda$ varies around $\bbT$.
We have already noted that the OPUC $\Phi_n$ is the characteristic polynomial of a cut-off CMV matrix $\mcgn$.
Using this parameter $\lambda \in \bbT$, we can characterize a family of rank one unitary dilations of $\mcgn$.
By adding one row and one column to $\mcgn$, we define a unitary $(n+1)\times (n+1)$ matrix whose characteristic polynomial is $\Phi_{n+1}^{(\lambda)}(z)$.
While the numerical range of $\mcgn$ has the $(n+1)$-Poncelet property, the numerical ranges of its unitary dilations are bounded by $(n+1)$-gons inscribed in $\bbT$ and circumscribed around $W(\mcgn)$.  The vertices of these $(n+1)$-gons are the eigenvalues of the matrices, or equivalently, the zeros of the POPUC.  
 
We have already mentioned that the boundary of the numerical range of $\mcgn$ has the $(n+1)$-Poncelet property.  This phenomenon can be reformulated as saying that $\partial W(\mcgn)$ is the envelope of the family of circumscribing $(n+1)$-gons.  The precise definition of an envelope and its relationship to numerical ranges is complicated, so we will restrict our attention only to the most relevant facts and refer the reader to \cite{HMFPS} for details.  The envelope is most easily understood by means of a dual curve.  If the matrix is $\mcgn$, then the dual curve is an algebraic curve of degree exactly $n$ and the dual of that dual is an algebraic curve of class $n$ with multiple components.  The largest component is the boundary of the numerical range of $\mcgn$ and we denote this component by $C_1$ (to be consistent with notation in \cite{HMFPS}).  The other components can be numbered $C_2,\ldots,C_{\lceil n/2\rceil}$ and they too have an interpretation in terms of the $(n+1)$-gons that circumscribe $\partial W(\mcgn)$.
     
To understand this interpretation, let us consider the component $C_2$, which we call the \textit{Pentagram curve} after the pentagram map from \cite{Pentagram}.  For each $(n+1)$-gon $P$ that circumscribes $\partial W(\mcgn)$, let us order the vertices cyclically on $\bbT$ by $P_1,P_2,\ldots,P_{n+1}$.  Consider now the ``polygon" obtained by joining $P_j$ to $P_{j+2}$ for every $j=1,\ldots,n+1$ (where arithmetic is done modulo $n+1$).  The resulting shape will be a non-convex $(n+1)$-gon if $n$ is even and it will be two $(n+1)/2$-gons with interlacing vertices if $n$ is odd.  In either case, one can define the envelope of this collection of polygons and that will be the curve $C_2$.  A similar construction yields the other curves $C_j$.  When $n=6$, we will refer to the curve $C_3$ as the \textit{Brianchon curve} (after Brianchon's Theorem  \cite[Theorem 5.4]{GBook}) and we note that the curve obtained from this procedure can be a single point (see Figures~\ref{fig:BrianchonpointBis} and 5).

It was Darboux who first proved that if the curve $C_1$ is an ellipse, then all of the other curves $C_j$ are ellipses, where we consider a single point to be a degenerate ellipse (see \cite{HMFPS}).  These ellipses all happen to be from the same package (see \cite{Mirman1}).  In the context of our motivating problem from the opening paragraph, if $C_1$ is an ellipse, then the foci of $C_1$ are eigenvalues of $\mcgn$.  Darboux's result implies that the other $C_j$ curves are ellipses, and it turns out that their foci are also eigenvalues of $\mcgn$.  It turns out that this property persists even if $C_1$ is not an ellipse.  More precisely, if any curve $C_j$ is an ellipse, then the foci of that ellipse are eigenvalues of $\mcgn$.  For this reason, finding matrices $\mcgn$ for which some components $C_j$ are ellipses is an interesting problem related to our primary objective, and we will present results of this kind in later sections (see Theorem \ref{newdecomp}, Theorem \ref{pentfix}, and Theorem \ref{brianfix} below).

We will also work with \textit{Blaschke products}
\begin{equation} \label{notation:Blaschke}
B_n(z):= \frac{ \Phi_{n}(z)}{\Phi_{n}^*(z)},
\end{equation}
where
\[
\Phi_{n}(z)= \prod_{j=1}^n \left(z-z_j \right),\qquad\qquad |z_j|<1.
\]
With this notation, we will say that the Blaschke product $B_n(z)$ has degree $n$.  If we need to make the dependence on the zeros explicit, then we will write $B_n(z;z_1,\ldots,z_n)$.  We will say that a Blaschke product $B_n(z)$ is \textit{regular} if $B_n(0)=0$.

Our discussion so far shows that the following sets are in bijection with one another:
\begin{enumerate}
\item[\rm{(i)}]  equivalence classes of matrices in $S_{n-1}$ (where equivalence is defined by unitary conjugation)
\item[\rm{(ii)}]  monic polynomials of degree $n-1$ with all of their zeros in $\bbD$
\item[\rm{(iii)}] regular degree $n$ Blaschke products
\item[\rm{(iv)}]  $\bbD^{n-1}$ (thought of as collections of Verblunsky coefficients)
\end{enumerate}
Much of what we will do in Sections \ref{sec:quadrilateral} and \ref{sec:hexagon} relates properties of the numerical range of a cutoff CMV matrix $\mcg^{(n-1)}$ to properties of the corresponding Blaschke product $z\Phi_{n-1}(z)/\Phi_{n-1}^*(z)$.  The next result will be very helpful in that regard.  It is essentially an OPUC version of a result from \cite{DGSSV}.

\begin{theorem}\label{newdecomp}
Let $n=jk$ and $B_n(z)$ be a regular Blaschke product, i.e. $\displaystyle B_n(z)=\frac{z\Phi_{n-1}(z)}{\Phi_{n-1}^{*}(z)}$. Then $B_n(z)$ can be expressed as a composition of two regular Blaschke products $B_j(B_k(z))$ (with the degree of $B_m$ equal to $m$) if and only if $\Phi_{n-1}(z)$ factors as
\[
\Phi_{n-1}(z)=\Phi_{k-1}(z)\prod_{m=1}^{j-1}\mcs_{\bar{a}_m}(\Phi_{k-1}(z))
\]
for some $\Phi_{k-1}$ having all of its zeros in $\bbD$ and some $\{a_1,\ldots,a_{j-1}\}\in\bbD^{j-1}$.  If this factorization holds, then the zeros of $\Phi_{k-1}$ are the zeros of $B_k(z)/z$ and $\{a_1,\ldots,a_{j-1}\}$ is the zero set of $B_j(z)/z$.
\end{theorem}
 
 \begin{proof}
 Let $B_n(z)=B_j(B_k(z))$. Then
 \[
 B_j=\frac{z(z-a_1)...(z-a_{j-1})}{(1-\overline{a_1}z)...(1-\overline{a_{j-1}}z)}
 \]
 for some $\{a_1,\ldots,a_{j-1}\}\in\bbD^{j-1}$. If
 \[
 B_k=\frac{z \Phi_{k-1}(z)}{\Phi_{k-1}^*(z)},
 \]
then
\begin{align*}
B_n(z)&=B_j\left(\frac{z \Phi_{k-1}(z)}{\Phi_{k-1}^*(z)}\right)\\
&=\frac{z \Phi_{k-1}(z)(z\Phi_{k-1}(z)-a_1\Phi_{k-1}^*(z))\cdots
(z\Phi_{k-1}(z)-a_{j-1}\Phi_{k-1}^*(z))}{\Phi_{k-1}^*(z)(\Phi_{k-1}^*(z)-\overline{a_1}z\Phi_{k-1}(z))\cdots
(\Phi_{k-1}^*(z)-\overline{a_{j-1}}z\Phi_{k-1}(z))} 
   \end{align*}
It follows that
\[
B_n(z)=\frac{z\Phi_{k-1}(z)\mcs_{\bar{a}_1}(\Phi_{k-1}(z))\cdots \mcs_{\bar{a}_{j-1}}(\Phi_{k-1}(z))}{\Phi_{k-1}^*(z)\mct_{\bar{a}_1}(\Phi_{k-1}^*(z))\cdots \mct_{\bar{a}_{j-1}}(\Phi_{k-1}^*(z))}
\]
and hence $\Phi_{n-1}$ has the desired factorization.  Reversing this reasoning shows the converse statement.
\end{proof}  

Our focus is on finding those $A\in S_{n-1}$ whose numerical range is bounded not just by an $n$-Poncelet curve but by an $n$-Poncelet ellipse.  The relationship between numerical ranges of $A\in S_{n-1}$ and Poncelet ellipses is a subject of significant ongoing research (see \cite{DGSSV,GBook,DGV,Fuj13,Fuj17,GW,HMFPS,Mirman1,MS}).  The following theorem (from \cite{MS}) is the starting point of our investigation and shows that the ellipses that we are looking for do in fact exist.  We state it using the terminology that we have defined so far.

 \begin{theorem}\label{Shukla}\cite{MS}
Suppose $f_1,f_2\in\bbD$.  There exists a Poncelet $n$-ellipse with foci at $f_1$ and $f_2$.  Furthermore, this ellipse forms the boundary of the numerical range of a matrix $A\in S_{n-1}$ and $f_1,f_2$ are eigenvalues of $A$.
 \end{theorem}
 
Our path forward is now clear.  Given $f_1,f_2\in\bbD$, we want to find a matrix $A\in S_{n-1}$ such that $\partial W(A)$ is an ellipse with foci at $f_1$ and $f_2$.  Theorem \ref{Shukla} tells us that such an $A$ exists, and we know that we can realize it as a cutoff CMV matrix.  Such a matrix has $n-1$ eigenvalues in $\bbD$, two of which must be $f_1$ and $f_2$.  A priori, there are no other restrictions on the remaining eigenvalues of $A$ other than they must be in $\bbD$.  In Section 3, we will show how to locate the third eigenvalue of $A$ when $n=4$ and in Section 4 we will consider the case when $n=6$ and find a set of algebraic equations in three variables whose unique solution marks the locations of the other three eigenvalues that we seek.

The most significant results known for general $n$ come from \cite{Mirman1,Mirman2,MS}.  The following theorem is a restatement of a result from \cite{Mirman1} using the terminology of matrices from $S_n$.
 
 \begin{theorem}\label{MirmanThm}
 Suppose $E_n$ is a Poncelet $n$-ellipse in $\bbD$ that is also the boundary of the numerical range of a matrix $A\in S_{n-1}$.  Suppose the foci of $E_n$ are $f_1$ and $f_2$.  Then the eigenvalues of $A$ can be labeled $\{w_1,\ldots,w_{n-1}\}$ so that $w_1=f_1$, $w_{n-1}=f_2$, and
 \begin{equation}\label{mirmans}
 w_{j-1}w_{j+1}=B_2(w_j;f_1,f_2),\qquad\qquad j=2,3,\ldots,n-2
 \end{equation}
 \end{theorem}
 
We will refer to the system of equations \eqref{mirmans} as the ``Mirman system''.  It gives us a necessary but not sufficient condition for a set of eigenvalues to be the spectrum of a matrix in $S_{n-1}$ whose numerical range is bounded by an ellipse with foci at $f_1$ and $f_2$.  We will see that this system of equations has many solutions and only one of them has the desired interpretation.  In Section 5, we will consider matrices in $S_4$ and discover some properties of all of the solutions to the Mirman system when $n=5$.

   % % % % % % % % % % % % % % % % % % % % % % % % % % % % % % % % % % % % % % % % % % % % % % % % % % % % % % % % % % % % % % 
  \section{The quadrilateral case} \label{sec:quadrilateral}
  % % % % % % % % % % % % % % % % % % % % % % % % % % % % % % % % % % % % % % % % % % % % % % % % % % % % % % % % % % % % % % 

In this section, we will consider matrices in $S_3$ to show how our approach to Poncelet ellipses using OPUC allows us to reformulate, and in some cases strengthen, existing results in the literature.  A classification of those matrices $A\in S_3$ for which $W(A)$ is an elliptical disk is given in \cite{Fuj13}.  Recall that a Blaschke product is regular if it maps $0$ to $0$.  Fujimura \cite{Fuj13} showed that $A\in S_3$ has a numerical range that is an elliptical disk if and only if there are regular Blaschke products $B_2,C_2$ of degree $2$ such that
\begin{equation}\label{3by3}
B_A(z):=\frac{z\det(zI-A)}{\det(zI-A)^*}=B_2(C_2(z))
\end{equation}
(see also \cite{DGSSV,GW}).  By Theorem \ref{newdecomp}, we can state the following theorem.

\begin{theorem}\label{3by3thm}
Suppose $A\in S_3$.  The numerical range of $A$ is an elliptical disk if and only if there exist $a,b\in\bbD$ so that
\[
\det(zI-A)=\Phi_1(z;a)\Phi_2(z;a,b).
\]
If this condition holds, then the pentagram curve is the single point $\bar{a}$ and the foci of $\partial W(A)$ are the zeros of $\Phi_2(z;a,b)$.
\end{theorem}

\textbf{Remark.}  Theorem \ref{3by3thm} implies that to any Poncelet $4$-ellipse $E\subseteq\bbD$, one can associate a well-defined point in $\bbD$ that will be the pentagram curve of the matrix $A\in S_3$ that satisfies $\partial W(A)=E$.  We will call this point the \textit{pentagram point} of the ellipse $E$.

\bigskip

\begin{figure}[h]
	\begin{center}
	\includegraphics[width=0.35\linewidth]{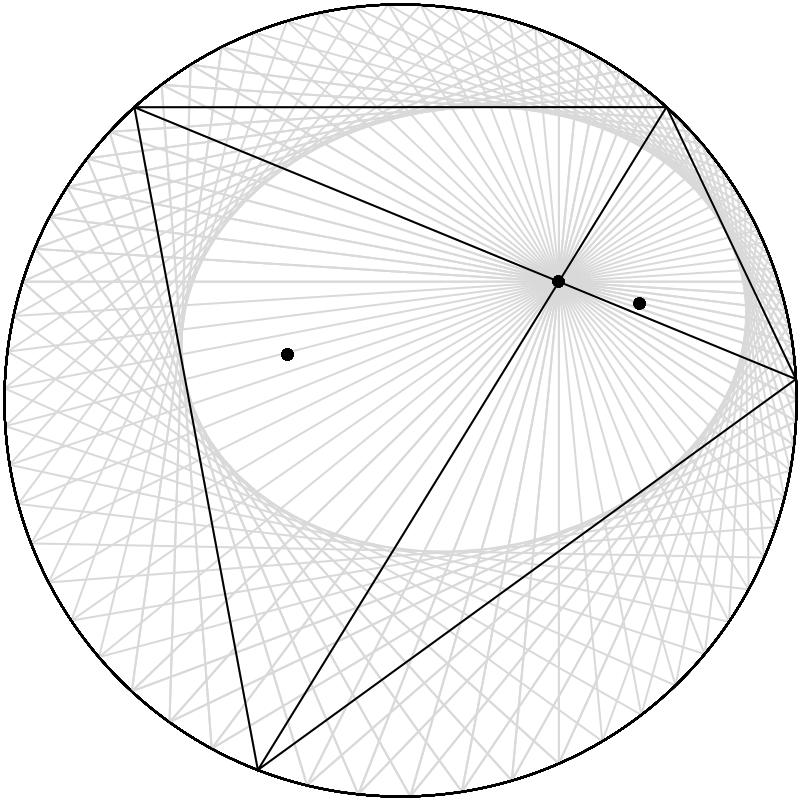}
	\caption{The Poncelet 4-ellipse with foci  $-0.28 + 0.12i$, $0.6  + 0.24 i$ and pentagram point $0.4+0.3i$.}
	\end{center}
	\label{fig:Poncelet4Ellpse}
\end{figure}

Theorem \ref{3by3thm} gives us a new interpretation of the algorithm for finding a matrix in $S_3$ with elliptic numerical range and two prescribed foci (such an algorithm can be found in \cite{GW}).  Indeed, given $\{f_1,f_2\}\in\bbD^2$, consider the polynomial $\Phi_2(z)=(z-f_1)(z-f_2)$.  Perform the Inverse Szeg\H{o} recursion to obtain a degree $1$ monic polynomial $\Phi_1(z;\bar{f}_3)$ whose zero $f_3$ is in $\bbD$.  The $3\times3$ cutoff CMV matrix with eigenvalues at $\{f_1,f_2,f_3\}$ has the desired property.

We can even generalize this algorithm to find an $A\in S_3$ with elliptic numerical range having one prescribed focus and a pentagram curve that is a prescribed point.

\begin{theorem}\label{3by3mixmatch}
Given $\{f_1,f_2\}\in\bbD^2$, there exists a unique $3\times3$ cutoff CMV matrix whose numerical range is bounded by an ellipse with one focus at $f_1$ and such that the pentagram curve is the single point $\{f_2\}$.
\end{theorem}

\begin{proof}
By Theorem \ref{3by3thm}, this amounts to showing that we can find $b\in\bbD$ such that $\Phi_2(z;\bar{f}_2,b)$ vanishes at $f_1$.  It is easy to see that this can be achieved precisely by setting
\[
\bar{b}=\frac{f_1\Phi_1(f_1;\bar{f}_2)}{\Phi_1^*(f_1;\bar{f}_2)}=\frac{f_1(f_1-f_2)}{1-\bar{f}_2f_1}
\]
(see also \cite{SimonTotik}).
\end{proof}

One can also find an $A_\alpha \in S_3$ with $\partial W(A_\alpha)$ an ellipse and whose pentagram curve is a specified point.  Since this is a weaker set of conditions than was used in Theorem \ref{3by3mixmatch}, one expects that we will have many solutions to this problem.  Rather than a unique matrix as in Theorem \ref{3by3mixmatch}, fixing the pentagram point yields a family of matrices in $S_3$ parametrized by the Verblunsky coefficient of $\Phi_2(z;a,b)$.

\begin{proposition}\label{quadpentfix}
Given $\alpha_0 \in \mathbb{D}$, there exists a one-parameter family $A_\alpha \in S_3, \hspace{1mm} \alpha \in \mathbb{D}$ such that for each $A_\alpha$, $\partial W(A_\alpha)$ is an ellipse and the pentagram point of $A_\alpha$ is $\alpha_0$.
\end{proposition}

\begin{proof}
Suppose $\alpha_0$ is given and consider the polynomial $\Phi_1(z;\bar{\alpha}_0)$. %maybe a note about notation from first paper?
For any $\alpha \in \mathbb{D}$, consider $\Phi_2(z;\bar{\alpha}_0,\alpha)$ and define
$$
\phi_3(z)=\Phi_1(z;\bar{\alpha}_0)\Phi_2(z;\bar{\alpha}_0,\alpha).
$$
Thinking of $\phi_3(z)$ as an OPUC and applying the inverse Szeg\H{o} recursion allows us to recover the Verblunsky coefficients of $\phi_3(z)$ and thus define a cutoff CMV matrix, $A_\alpha \in S_3$, whose characteristic polynomial is $\phi_3(z)$. 
As $\phi_3(z)$ factors into a degree one and degree two OPUC related by the Szeg\H{o} recursion, Theorem \ref{3by3thm} implies that $\partial W(A_\alpha)$ is an ellipse and the pentagram point of $A_\alpha$ is $\alpha_0$.
\end{proof}

\begin{figure}[htb] 
	\begin{center}
		\begin{tabular}{c p{10mm}c}		
			\includegraphics[width=0.35\linewidth]{Fig_foci2} 
			& \ 	& \includegraphics[width=0.35\linewidth]{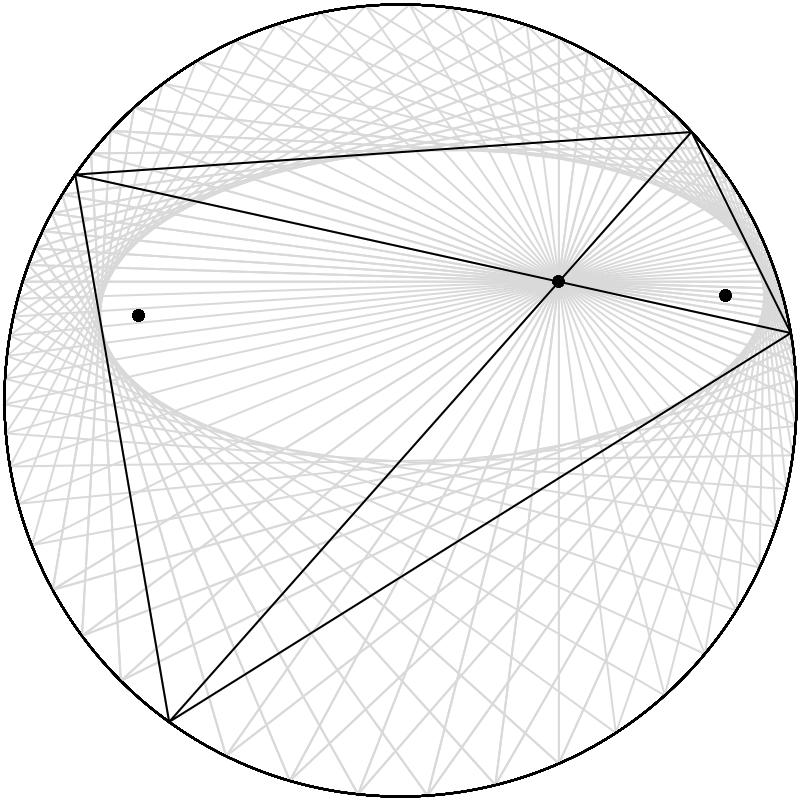}
	\end{tabular}
	\end{center}
	\caption{Two Poncelet 4-ellipses with the same pentagram point, $0.4+0.3i$. }
	\label{fig:quadsamepent}
\end{figure}

Now suppose we have an $A\in S_3$ with numerical range given by an elliptical disk.  Suppose that the foci of the ellipse bounding that elliptical disk are $\{f_1,f_2\}$ and the pentagram point of that ellipse is $\{f_3\}$.  Then
\[
\det(zI-A)=(z-f_1)(z-f_2)(z-f_3).
\]
By Theorem \ref{3by3thm}, we also have
\[
\det(zI-A)=(z-f_3)(z(z-f_3)-b(1-\bar{f}_3z))
\]
for some $b\in\bbD$.  Evaluating both of these expressions at $0$ and equating them shows $b=-f_1f_2$.  It follows that
\begin{equation}\label{3by3-1}
(z-f_1)(z-f_2)=z(z-f_3)+f_1f_2(1-\bar{f}_3z)=z^2+z(-f_3-f_1f_2\bar{f}_3)+f_1f_2.
\end{equation}
If we replace $z$ by $f_3$ in \eqref{3by3-1}, we get
\begin{equation}\label{3by3-2}
(f_3-f_1)(f_3-f_2)=f_1f_2(1-|f_3|^2).
\end{equation}
If we look at the reversed polynomials in \eqref{3by3-1} and replace $z$ by $f_3$, we get
\begin{equation}\label{3by3-3}
(1-\bar{f}_1f_3)(1-\bar{f}_2f_3)=1-|f_3|^2.
\end{equation}
If we divide \eqref{3by3-2} by \eqref{3by3-3}, we recover the Mirman system: $f_1f_2=B_2(f_3;f_1,f_2)$.
Solving for $f_3$ yields the familiar formula for $f_3$ in terms of  $f_{1},f_{2}$:
$$
f_3= \frac{f_1+f_2 - f_2 |f_1|^2 - f_1 |f_2|^2}{1 -  |f_1 f_2|^2}.
$$
Notice that we have recovered something that the Mirman system does not give us.  One can verify that $f_3=0$ is a solution to the Mirman system, but this does not (in general) give us the matrix in $S_3$ with elliptical numerical range.  Our calculations using OPUC eliminate this extraneous solution.

  % % % % % % % % % % % % % % % % % % % % % % % % % % % % % % % % % % % % % % % % % % % % % % % % % % % % % % % % % % % % % % 
  \section{The hexagon case} \label{sec:hexagon}
  % % % % % % % % % % % % % % % % % % % % % % % % % % % % % % % % % % % % % % % % % % % % % % % % % % % % % % % % % % % % % % 

In this section we will consider curves that are realized as the envelope of line segments joining vertices of hexagons inscribed in the unit circle (see \cite[Section 3]{HMFPS} for a rigorous discussion of envelopes of cyclic polygons).  We know that such curves have three components: the largest one (outer) formed by connecting adjacent eigenvalues of the unitary dilations is the \textit{Poncelet curve}, the middle component formed by joining alternate eigenvalues is the \textit{pentagram curve}, and the smallest component formed by joining opposite eigenvalues is the \textit{Brianchon curve}.  Our first results are the following two theorems.
\begin{theorem}\label{pent6}
	Let $\mcg^{(5)}$ be a cut-off CMV matrix and 	
	\begin{equation*}
	\Phi_5(z):=\det (z I_5 -\mcg^{(5)}).
	\end{equation*}
	The following are equivalent.
	\begin{enumerate}
		\item[\rm{(i)}] the pentagram curve of $\mcg^{(5)}$ is an ellipse;
		\item[\rm{(ii)}] there exist regular Blaschke products $\{B_j\}_{j=2}^3$ with $\deg(B_j)=j$ such that
		\[
		\frac{z\Phi_5(z)}{\Phi_5^*(z)}=B_2(B_3(z))
		\]
		\item[\rm{(iii)}] There exist $\alpha_0,\alpha_1,\alpha_2\in\mathbb D$ so that
		\[
		\Phi_5(z)=\Phi_2(z;\alpha_0,\alpha_1) \Phi_3(z;\alpha_0,\alpha_1,\alpha_2) . %\Phi_2(z;\alpha_0,\alpha_1)\Phi_3(z;\alpha_0,\alpha_1,\alpha_2)
		\]
	\end{enumerate}
	If any of these conditions hold, then the foci of the pentagram curve are the zeros of $B_3(z)/z$ or equivalently, the zeros of $ \Phi_2(z;\alpha_0,\alpha_1) $.
\end{theorem}

\begin{figure}[htb]
	\begin{center}
	\includegraphics[width=0.35\linewidth]{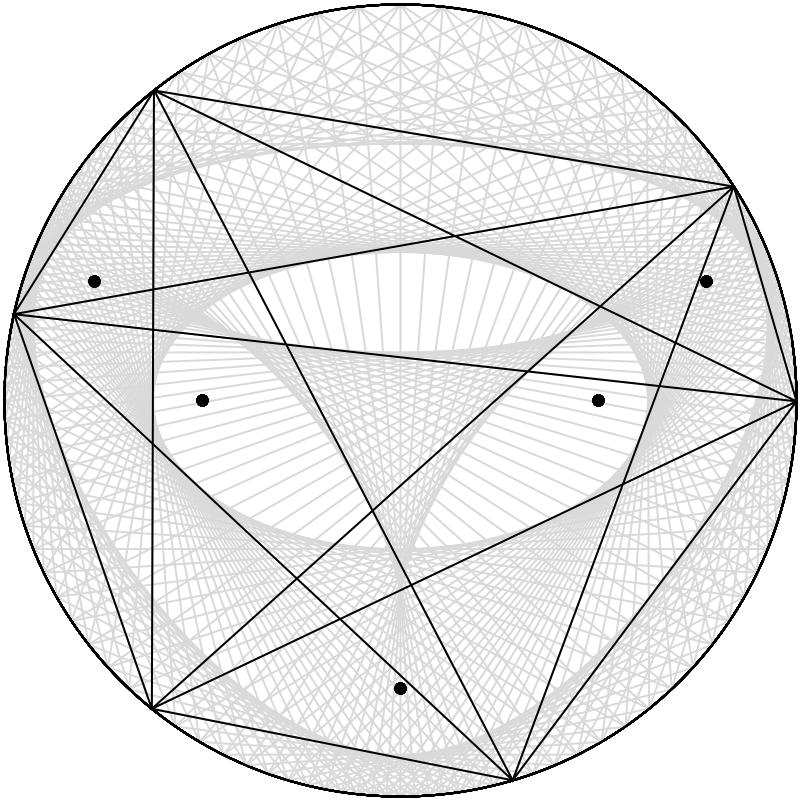}
	\caption{A $6$-Poncelet curve such that the  pentagram curve is an ellipse but the Brianchon curve is not a single point.}
	\end{center}
	\label{fig:BrianchonpointBis}
\end{figure}

\begin{theorem}\label{brianch6}
	If we retain the notation from Theorem \ref{pent6}, then the following are equivalent:
	\begin{enumerate}
		\item[\rm{(i)}] the Brianchon curve of $\mcg^{(5)}$ is a single point;
		\item[\rm{(ii)}] there exist regular Blaschke products $\{B_j\}_{j=2}^3$ with $\deg(B_j)=j$ such that
		\[
		\frac{z\Phi_5(z)}{\Phi_5^*(z)}=B_3(B_2(z))
		\]
		\item[\rm{(iii)}] $\,$There exist $\alpha_0,\alpha_1,\gamma_1\in\mathbb D$ so that
		\[
		\Phi_5(z)=\Phi_1(z;\alpha_0)\Phi_2(z;\alpha_0,\alpha_1)\Phi_2(z;\alpha_0,\gamma_1)
		\]
	\end{enumerate}
	If any of these conditions hold, then the Brianchon point is the zero of $B_2(z)/z$ and the zero of $\Phi_1(z;\alpha_0)$.
\end{theorem}

The equivalence of (i) and (ii) in both theorems is proven in \cite{DGSSV}.  The equivalence of (ii) and (iii) in both theorems follows from Theorem \ref{newdecomp}.  By comparing Theorem \ref{brianch6} with Theorem \ref{3by3thm} (and the remark after it), one arrives at the following result.
	
	\begin{corollary}\label{46relation}
		Let $A\in S_5$ have eigenvalues $\{f_j\}_{j=1}^5$.  Suppose the Brianchon curve of $A$ is the point $f_5$.  Then $\{f_j\}_{j=1}^4$ can be labelled in such a way that both of the following conditions hold:
		\begin{enumerate}
			\item[\rm{(i)}] $f_5$ is the pentagram point of the Poncelet $4$-ellipse with foci at $f_1$ and $f_4$;
			\item[\rm{(ii)}] $f_5$ is the pentagram point of the Poncelet $4$-ellipse with foci at $f_2$ and $f_3$.
		\end{enumerate}
	\end{corollary}

\medskip

Using ideas from \cite{HMFPS}, we can prove the following.

\begin{theorem}\label{ponc6}
	If we retain the notation from Theorem \ref{pent6}, then the following are equivalent
	\begin{enumerate}
		\item[\rm{(i)}] The Poncelet curve associated with $\mcg^{(5)}$ is an ellipse.
		\item[\rm{(ii)}]There exist regular Blaschke products $\{B_j\}_{j=2}^3$ with $\deg(B_j)=j$ and regular Blaschke products $\{C_j\}_{j=2}^3$ with $\deg(C_j)=j$ such that
		\[
		\frac{z\Phi_5(z)}{\Phi_5^*(z)}=B_3(B_2(z))=C_2(C_3(z)).
		\]
		\item[\rm{(iii)}] The pentagram curve of $\mcg^{(5)}$ is an ellipse and the Brianchon curve of $\mcg^{(5)}$ is a single point.
	\end{enumerate}
\end{theorem}

\begin{figure}[htb]
	\begin{center}
	\includegraphics[width=0.35\linewidth]{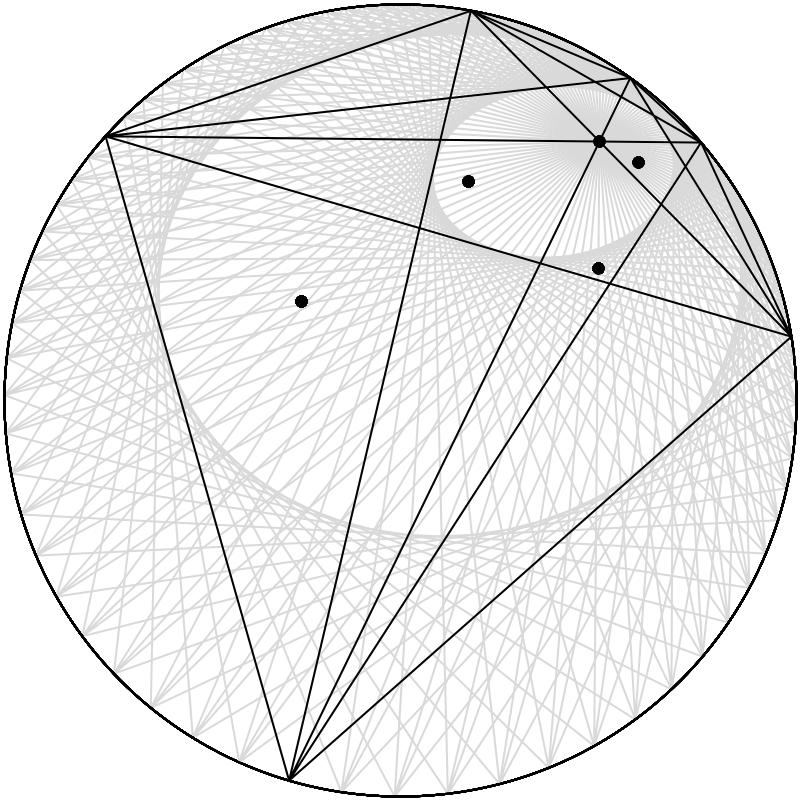}
	\caption{A Poncelet 6-ellipse along with the pentagram ellipse and Brianchon point.}
	\end{center}
	\label{fig:Poncelet6Ellpse}
\end{figure}

\begin{proof}
	The equivalence of (ii) and (iii) is an immediate consequence of Theorems \ref{pent6} and \ref{brianch6}.  The fact that (i) implies (iii) is a result of Darboux (see \cite[Theorem B]{HMFPS}).  To see that (iii) implies (i), we will make use of the dual curve described in Section \ref{sec:background}.  Notice that \cite[Section 3]{HMFPS} implies that if $C_1$, $C_2$, or $C_3$ is an algebraic curve of degree $1$ or $2$, then the same is true of the corresponding component of the dual curve and vice versa (counting a single point as having degree $1$).  The dual curve in this case has degree $5$ (see \cite[Section 3]{Fuj18}).  Thus, if the Brianchon curve has degree $1$ and the pentagram curve has degree $2$, then the Poncelet curve has degree $2$, which means it is an ellipse.
\end{proof}

Theorem \ref{brianch6} reveals an interesting phenomenon.  Suppose we are given an $A\in S_5$ whose Brianchon curve is a point and whose pentagram curve is not an ellipse (it is easy to find such $A$).  Take a rank one unitary dilation $U_1$ of $A$ and look at the hexagon with vertices at the eigenvalues of $U_1$.  The diagonals of this hexagon meet in a single point (which is the Brianchon curve of $A$), so by Brianchon's Theorem there exists an ellipse $E_1$ inscribed in this hexagon.  By construction, the ellipse $E_1$ is a Poncelet $6$-ellipse, so there is in fact an infinite family of hexagons $\{H_{\lambda}\}$ inscribed in $\partial\mathbb D$ and circumscribed about $E_1$.  On the other hand, one can look at any other rank one unitary dilation of $A$ and repeat this process.  This gives a second infinite family of hexagons, each of which is circumscribed in $\partial\mathbb D$ and has its diagonals meeting at the point that is the Brianchon curve of $A$.  But these two families of hexagons cannot be the same, for that would imply that the numerical range of $A$ is bounded by an ellipse and we know it is not.

\begin{figure}[htb] 
	\begin{center}
		\begin{tabular}{c p{10mm}c}		
			\includegraphics[width=0.30\linewidth]{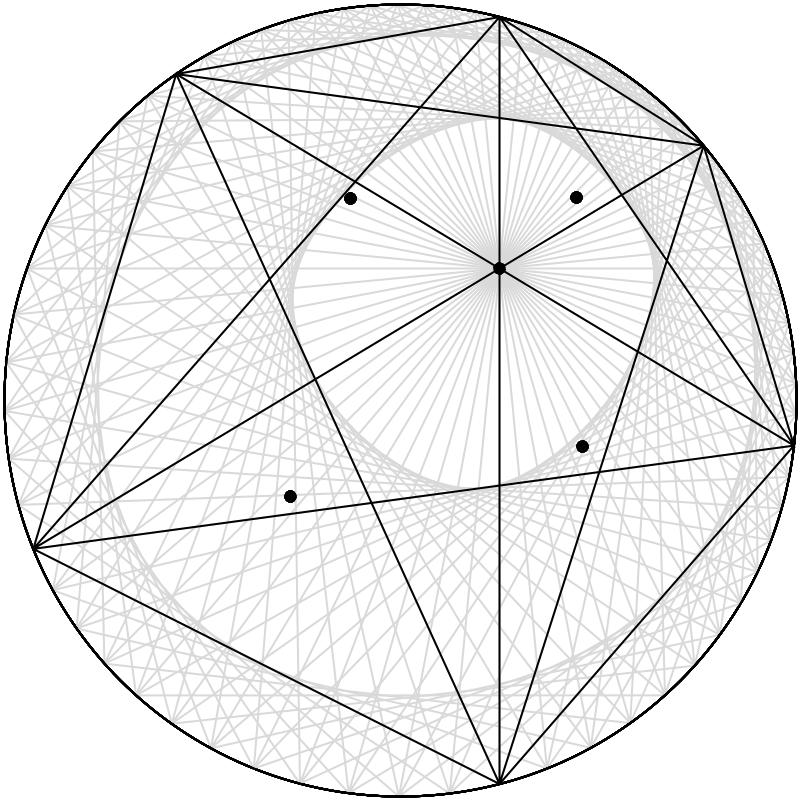} 
			& \ 	& \includegraphics[width=0.30\linewidth]{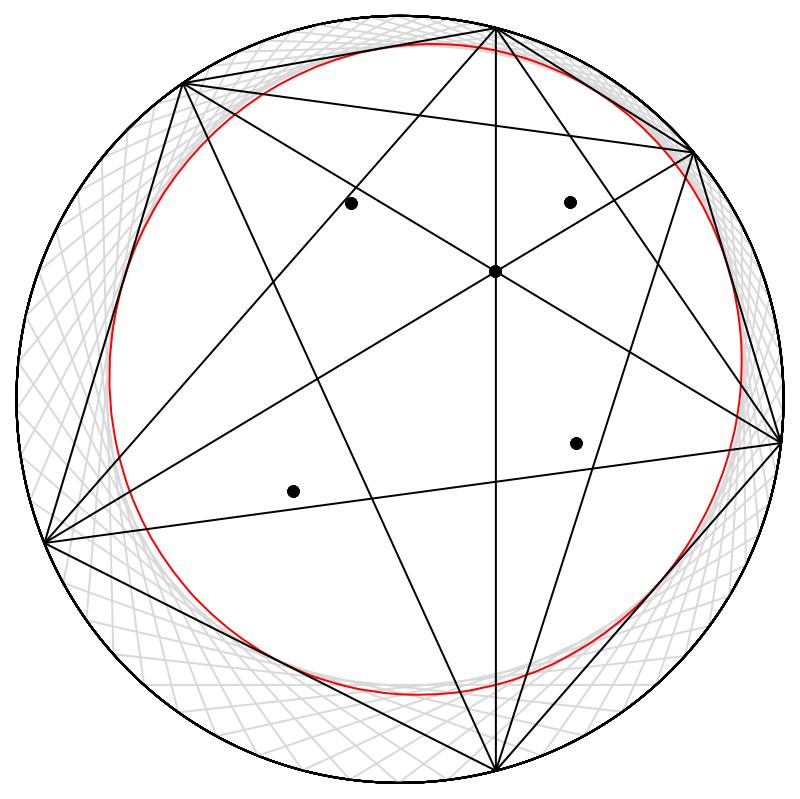}
	\end{tabular}
	\end{center}
	\caption{A $6$-Poncelet curve that is not an ellipse having the property that its Brianchon curve is a single point (Brianchon point). The picture on the right shows the ellipse that is inscribed in one of the Poncelet hexagons. This ellipse exists by Brianchon's
		theorem. Notice that the $6$-Poncelet curve contains points in the exterior as well as points in the interior of this ellipse.}
	\label{fig:Brianchonpoint}
\end{figure}

The next step in our analysis will be very much analogous to that performed in Section \ref{sec:quadrilateral}.  If we are given $\{f_1,f_2\}\in\bbD^2$, we want to find $A\in S_5$ whose numerical range is bounded by an ellipse with foci at $\{f_1,f_2\}$.  Theorem \ref{Shukla} tells us that such an ellipse exists and is the unique Poncelet $5$-ellipse with foci at $f_1$ and $f_2$.  We will find a cutoff CMV matrix $\mcg^{(5)}$ whose numerical range is bounded by an ellipse with foci at $\{f_1,f_2\}$, whose pentagram curve is an ellipse (whose foci will be called $\{f_3,f_4\}$), and whose Brianchon curve is a single point (that we will call $f_5$).

\begin{theorem}\label{5by5thm}
A matrix $A\in S_5$ with eigenvalues $\{f_j\}_{j=1}^5$ satisfies all of the following conditions
\begin{enumerate}
\item[\rm{(i)}] $W(A)$ is an elliptical disk
\item[\rm{(ii)}] the foci of $\partial W(A)$ are $\{f_1,f_2\}$
\item[\rm{(iii)}] the foci of the pentagram curve are $\{f_3,f_4\}$
\item[\rm{(iv)}] the Brianchon curve is the single point $f_5$
\end{enumerate}
if and only if the complex numbers $\{f_j\}_{j=1}^5$ satisfy
\begin{align}\label{f34}
	\{f_3,f_4\}=\left\{\frac{f_5-f_2}{1-f_2\bar{f}_5},\frac{f_5-f_1}{1-f_1\bar{f}_5}\right\}
	\end{align}
as sets and
\begin{equation}\label{f12345}
	\begin{aligned}
	f_3+f_4+f_1f_2f_5\bar{f}_4\bar{f}_3&=f_1+f_2+f_5\\
	f_3f_4+f_1f_2f_5(\bar{f}_4+\bar{f}_3)&=f_1f_2+f_2f_5+f_1f_5
	\end{aligned}
	\end{equation}
\end{theorem}

Recall that Theorem \ref{brianch6} and Theorem \ref{ponc6} show that if $W(A)$ is bounded by an ellipse, then the characteristic polynomial factors in a certain way (in terms of OPUC) and the degree $1$ polynomial in that factorization vanishes at the Brianchon point.  We require the following refinement of that result, which tells us about the zeros of the remaining polynomials in that factorization.

\begin{lemma}\label{5zeros}
Suppose $A\in S_5$ is such that $W(A)$ is an elliptical disk and the eigenvalues of $A$ are $\{f_j\}_{j=1}^5$.  Suppose the foci of $\partial W(A)$ are $\{f_1,f_2\}$, the foci of the pentagram curve are $\{f_3,f_4\}$, and the Brianchon point is $f_5$.  Write
\[
\det(zI_5-A)=\Phi_1(z;\alpha_0)\Phi_2(z;\alpha_0,\alpha_1)\Phi_2(z;\alpha_0,\gamma_1)
\]
as in Theorem \ref{brianch6}.  If $\Phi_2(f_1;\alpha_0,\alpha_1)=0$, then $\Phi_2(f_2;\alpha_0,\gamma_1)=0$.
\end{lemma}

\begin{proof}
Suppose $\Phi_2(f_1;\alpha_0,\alpha_1)=0$ and $\Phi_2(f_2;\alpha_0,\gamma_1)\neq0$.  Then $\Phi_2(f_2;\alpha_0,\alpha_1)=0$ and hence $\Phi_2(f_3;\alpha_0,\gamma_1)=\Phi_2(f_4;\alpha_0,\gamma_1)=0$. 

Since $\Phi_2(f_2;\alpha_0,\alpha_1)=0$, we can write
\[
(z-f_5)(z-f_1)(z-f_2)=\Phi_1(z;\alpha_0)\Phi_2(z;\alpha_0,\alpha_1).
\]
This means $\{f_1,f_2,f_5\}$ are the eigenvalues of some $A\in S_3$ that satisfies the hypotheses of Theorem \ref{3by3thm}.  Applying the Mirman system in the $N=4$ case shows
\[
f_1f_2=B_2(f_5;f_1,f_2) 
\]
The Mirman system in the case $n=6$ shows  shows $f_3f_4=B_2(f_5;f_1,f_2)$ and hence $f_1f_2=f_3f_4$.  Since $f_3,f_4$ are the zeros of $\Phi_2(z;\alpha_0,\gamma_1)$ and $\gamma_1=-\overline{\Phi_2(0;\alpha_0,\gamma_1)}$, we conclude that $\gamma_1=\alpha_1$, which implies $\Phi_2(z;\alpha_0,\gamma_1)=\Phi_2(z;\alpha_0,\alpha_1)$.  It follows that $\Phi_2(f_2;\alpha_0,\gamma_1)=0$, which gives us a contradiction.
\end{proof}

\textit{Proof of Theorem \ref{5by5thm}}
In Theorem \ref{pent6} it is stated that the foci of the pentagram ellipse will be the zeros of $\Phi_2(z;\alpha_0,\alpha_1)$.  Thus, the foci of the Poncelet curve and the Brianchon point must be the zeros of $\Phi_3(z;\alpha_0,\alpha_1,\alpha_2)$.  The product of these zeros is then $\bar{\alpha}_2$ and hence we have
	\begin{equation}\label{3rel}
	z(z-f_3)(z-f_4)-f_1f_2f_5(1-\bar{f_3}z)(1-\bar{f}_4z)=(z-f_1)(z-f_2)(z-f_5)
	\end{equation}
Equating coefficients of $z$ and $z^2$  in \eqref{3rel} tells us that \eqref{f12345} must hold.  One can perform a similar calculation invoking Theorem \ref{brianch6}.  By equating coefficients of the appropriate polynomials, we find that \eqref{f34} must hold.

For the converse statement, the above calculations show that if \eqref{f34} and \eqref{f12345} hold, then the conditions (iii) in Theorems \ref{pent6} and \ref{brianch6} are satisfied (by equating coefficients of polynomials).  Theorem \ref{ponc6} then implies that the cutoff CMV matrix $\mcg^{(5)}$ with eigenvalues $\{f_j\}_{j=1}^5$ has numerical range that is bounded by an ellipse, has pentagram curve that is an ellipse with foci $\{f_3,f_4\}$, and has Brianchon point $\{f_5\}$.  The foci of $\partial W(\mcg^{(5)})$ are eigenvalues of $\mcg^{(5)}$ (see \cite[Section 5]{HMFPS}).  By \cite[Corollary 4]{Mirman1}, we know that one can partition the eigenvalues of $\mcg^{(5)}$ into the Brianchon point, the foci of the pentagram curve, and the foci of the boundary of the numerical range.  Thus, by elimination it must be that the foci of $\partial W(\mcg^{(5)})$ are $\{f_1,f_2\}$.
\begin{flushright}$\square$\end{flushright}
 
Recall that the Mirman system in the $N=6$ case is
\begin{equation}\label{Mir6}
	f_1f_5=B_2(f_3;f_1,f_2),\qquad\qquad f_3f_4=B_2(f_5;f_1,f_2),\qquad\qquad f_2f_5=B_2(f_4;f_1,f_2),
\end{equation}
The conditions \eqref{f34} and \eqref{f12345} allow us to recover these relations.  Substitute $z$ for $f_3$ in \eqref{3rel} to obtain
\begin{equation}\label{Mir6Der1}
(f_3-f_1)(f_3-f_2)=\frac{f_1f_2f_5(1-|f_3|^2)(1-f_3\bar{f}_4)}{f_5-f_3}
\end{equation}
If we replace $z$ by $f_3$ in the reversed polynomials from \eqref{3rel}, then we obtain
\begin{equation}\label{Mir6Der2}
(1-\bar{f}_1f_3)(1-\bar{f}_2f_3)=\frac{(1-|f_3|^2)(1-f_3\bar{f}_4)}{1-\bar{f}_5f_3}
\end{equation}
If we divide \eqref{Mir6Der1} by \eqref{Mir6Der2}, and use \eqref{f34}, we find $B_2(f_3;f_1,f_2)=f_1f_5$.  Similar reasoning can be used to derive $B_2(f_4;f_1,f_2)=f_2f_5$.  Replacing $z$ by $f_5$ in \eqref{3rel} tells us that
	\[
	\frac{f_5(f_5-f_4)(f_5-f_3)}{(1-\bar{f}_4f_5)(1-\bar{f}_3f_5)}=f_1f_2f_5
	\]
	If we assume $f_1f_2f_5\neq0$ and we substitute the relations (\ref{f34}) for $f_3$ and $f_4$ (for an appropriate choice of which to call $f_3$ and which to call $f_4$), then we find
	\[
	(1-\bar{f}_2f_5)(1-\bar{f}_1f_5)=(1-\bar{f}_5f_2)(1-\bar{f}_5f_1)
	\]
	In other words $(1-\bar{f}_2f_5)(1-\bar{f}_1f_5)\in\bbR$.  If we use this fact, then multiplying the expressions in (\ref{f34}) gives $B_2(f_5;f_1,f_2)=f_3f_4$.
	
\medskip

If one is given $(f_1,f_2)\in\bbD^2$, the construction above shows how to find an $A\in S_5$ so that $\partial W(A)$ is an ellipse with foci at $f_1$ and $f_2$.  Our next result shows that one can similarly find $A\in S_5$ with $\partial W(A)$ an ellipse and whose pentagram curve has prescribed foci (Theorem \ref{ponc6} assures us that if the Poncelet curve of $A$ is an ellipse, then so is its pentagram curve).

\begin{theorem}\label{pentfix}
Given $(f_3,f_4)\in\bbD^2$, there exists a unique (up to unitary conjugation) $A\in S_5$ so that $\partial W(A)$ is an ellipse and the pentagram curve of $A$ is an ellipse with foci at $f_3$ and $f_4$.
\end{theorem}

Our proof of Theorem \ref{pentfix} requires the following two lemmas.

\begin{lemma}\label{unique}
Given two triangles inscribed in $\partial\bbD$ with interlacing vertices, there is a unique ellipse that is inscribed in both of them.
\end{lemma}

\begin{proof}
Any such ellipse would be a Poncelet $3$-ellipse.  The lemma is a consequence of Wendroff's Theorem for POPUC, which includes a uniqueness statement (see \cite[Theorem 3.1]{GauWu04}, \cite[Theorem 8]{MFSS}, or \cite{AGT}).
\end{proof}

\begin{lemma}\label{intersection}
Let $\{\Phi_3^{(\lambda)}\}_{\lambda\in\partial\bbD}$ be the collection of degree $3$ POPUC for the same degree $2$ OPUC.  Label the zeros of $\Phi_3^{(\lambda)}$ as $\{z_j^{(\lambda)}\}_{j=1}^3$.  For each $\lambda\in\partial\bbD$ there exists a unique $\tau\in\partial\bbD$ so that the line segments joining $z_j^{(\lambda)}$ to $z_j^{(\tau)}$ all meet in a single point independent of $j$ (this assumes an appropriate labeling of the zeros $\{z_j^{(\lambda)}\}_{j=1}^3$).
\end{lemma}

\begin{proof}
For each $\tau$, let $z_1^{(\tau)}$ be the zero that lies between $z_2^{(\lambda)}$ and $z_3^{(\lambda)}$, let $z_2^{(\tau)}$ be the zero that lies between $z_3^{(\lambda)}$ and $z_1^{(\lambda)}$, and let $z_3^{(\tau)}$ be the zero that lies between $z_1^{(\lambda)}$ and $z_2^{(\lambda)}$.  Let $L_j^{(\tau)}$ be the line segment that joins $z_j^{(\lambda)}$ to $z_j^{(\tau)}$ for $j=1,2,3$.

Given $\{z_j^{(\lambda)}\}_{j=1}^3$, start with $\tau=\lambda$ and move $\tau$ around $\bbT$ counterclockwise.  As this happens, consider $\eta_j(\tau):=L_1^{(\tau)}\cap L_j^{(\tau)}$ for $j=2,3$ and notice that both of these points are in $L_1^{(\tau)}$.  Define
\[
d_j(\tau)=\frac{|z_1^{(\lambda)}-\eta_j(\tau)|}{|L_1^{(\tau)}|},\qquad\qquad j=2,3.
\]
Initially, $d_3(\tau)$ is close to $0$ and $d_2(\tau)$ is close to $1$.  As $\tau$ nears the end of its trip around $\bbT$, it holds that $d_2(\tau)$ is close to $0$ and $d_3(\tau)$ is close to $1$.  Thus, by the Intermediate Value Theorem, there must be a value of $\tau$ such that $d_2(\tau)=d_3(\tau)$ as desired.  One can see by inspection that this choice of $\tau$ is unique.
\end{proof}

\textit{Proof of Theorem \ref{pentfix}}
Suppose $(f_3,f_4)\in\bbD^2$ are given.  Consider the Poncelet $3$-ellipse with foci at $f_3$ and $f_4$ (call it $E$).  Pick any triangle $T^{(\lambda)}$ that is inscribed in $\partial\bbD$ and circumscribed about $E$.  By Lemma \ref{intersection}, there exists a unique second such triangle $T^{(\tau)}$ such that the line segments joining opposite vertices meet in a single point.  By Brianchon's Theorem, there is an ellipse inscribed in the hexagon whose vertices are the vertices of $T^{(\tau)}$ and $T^{(\lambda)}$.  Call this ellipse $E'$ and suppose its foci are $f_1$ and $f_2$.

From what we already know, the ellipse $E'$ is the unique Poncelet $6$-ellipse with foci at $f_1$ and $f_2$, and Theorem \ref{Shukla} tells us that it is associated to a matrix in $S_5$.  For this matrix, the associated pentagram curve must be an ellipse and the associated Brianchon curve must be a single point.  This pentagram ellipse is a Poncelet $3$-ellipse and must be inscribed in the triangles $T^{(\lambda)}$ and $T^{(\tau)}$.  By Lemma \ref{unique}, that ellipse must be $E$.
\begin{flushright}$\square$\end{flushright}

Our next result is an analog of  Theorem \ref{pentfix} for the Brianchon curve.  Specifically, we will show that one can find $A \in S_5$ so that $\partial W(A)$ is an ellipse and the Brianchon curve is a single predetermined point.  The main difference between Theorem \ref{brianfix} and Theorem \ref{pentfix} is the lack of uniqueness.

\begin{theorem}\label{brianfix}
Given $f_5 \in \mathbb{D}$, there exists a $5\times5$ cutoff CMV matrix $A$ so that $\partial W(A)$ is an ellipse and the Brianchon curve of $A$ is the single point $f_5$.  Furthermore, the set of all possible $5\times5$ cutoff CMV matrices with this property is naturally parametrized by an open triangle.
\end{theorem}

\begin{proof}
Suppose $f_5 \in \mathbb{D}$ is given.  Consider the set of all lines passing through $f_5$.  Each line intersects $\mathbb{T}$ in two places.  One can choose three distinct lines, thus specifying 6 distinct points of $\mathbb{T}$, labeled cyclically as $v_j \text{ for } j=1,2,...,6$.  By Brianchon's Theorem, there is an ellipse inscribed in the hexagon whose vertices are $\{v_j\}_{j=1}^6$.  Call this ellipse $E$ and its foci $f_1 \text{ and }f_2$. By Theorem \ref{5by5thm}, $E$ is the boundary of $W(A)$ for some $A \in S_5$.  Thus, the Poncelet curve and pentagram curve of $A$ are ellipses and the Brianchon curve of $A$ is a single point, which we see must be $f_5$ as desired.

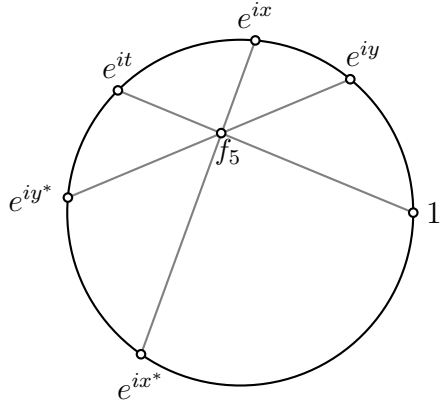
\begin{figure}[htb]
	\begin{center}
	 \begin{tikzpicture}[scale=2.3,cap=round,>=latex]
	 	% Preamble
	 	\tikzset{mypoints/.style={fill=white,draw=black,thick}}
	 	\def\ptsize{0.7pt}
	 	
	 	% draw the unit circle
	 	\draw[thick] (0cm,0cm) circle(1cm);
	 	
	 	% Define a pair of points and join
	 	\coordinate (A) at (1cm,0cm);
	 	\coordinate (B) at (135:1cm);
	\draw[gray,thick] (A) -- (B);
%	\filldraw[black] (135:1cm) circle(0.4pt);
	\draw (B) node[above=1pt] {\large $e^{it}$};
%	\filldraw[black] (1cm,0cm) circle(0.4pt);
	\draw (A) node[right=1pt] {\large $1$};
	
		% Define the second pair of points and join
	\coordinate (C) at (85:1cm);
	\coordinate (D) at (235:1cm);
	\draw[gray,thick] (C) -- (D);
%	\filldraw[black] (85:1cm) circle(0.4pt);
%	\filldraw[black] (235:1cm) circle(0.4pt);
	\draw (C) node[above=1pt] {\large $e^{ix}$};
	\draw (D) node[below=3pt] {\large $e^{ix^*}$};
	
	% Define the third pair of points and join
	\coordinate (E) at (175:1cm);
	\coordinate (F) at (50.5:1cm);
	\draw[gray,thick] (E) -- (F);
%	\filldraw[black] (175:1cm) circle(0.4pt);
%	\filldraw[black] (50.5:1cm) circle(0.4pt);
	\draw (F) node[right=5pt,above=1pt] {\large $e^{iy}$};
	\draw (E) node[left=1pt] {\large $e^{iy^*}$};
	
	% The intersection point
	\coordinate (G) at (-0.11cm,0.46cm);
	\node (a) at (-0.07cm,0.35cm) {$f_5$};
	
	% Draw the small circles
	\foreach \p in {A,B,C,D,E,F,G}
	\fill[mypoints] (\p) circle (\ptsize);
	
\end{tikzpicture}
	\caption{Given $f_5$ and the line through 1,$f_5$, varying $x,y$ such that $0<x<y<t$ parametrizes the set of all $A \in S_5$ with elliptic numerical range and Brianchon curve $f_5$.}
	\end{center}
	\label{fig:fixBrianchonpoint}
\end{figure}

Note that the above procedure yields a well-defined map from collections of distinct triples of lines through $f_5$ to matrices in $S_5$ whose numerical range is an elliptical disk and whose Brianchon point is $f_5$.  To see this, suppose one starts with three distinct lines through $f_5$.  This produces six points on the unit circle by the above procedure. Connecting alternate points forms two triangles that must be tangent to the pentagram ellipse of the matrix we seek.  By Lemma \ref{unique}, there is only one possible choice for such an ellipse, so its foci $f_3,f_4$ are a well-defined output.  By Theorem \ref{5by5thm} and equation \eqref{f34}, we can calculate the foci  $f_1, f_2$ of the Poncelet curve of this matrix.   Thus, the eigenvalues $\{f_j\}_{j=1}^5$ of the matrix we seek are a computable quantity from any three distinct lines through $f_5$.  Since the eigenvalues determine the matrix in $S_5$, this means the map is well-defined.

To make this map a bijection, we restrict it to triples of lines through $f_5$ for which one of them also passes through $1$.  Given any $A\in S_5$ with $W(A)$ an elliptical disk and Brianchon point $f_5$, there exists a hexagon that circumscribes $\partial W(A)$ for which $1$ is a vertex, so the restricted map is onto.  To show that it is injective, suppose $H$ is a hexagon inscribed in $\bbT$ with one vertex at $1$ and $E$ is an ellipse that is tangent to every edge of $H$.  From any given point on $\bbT$ (in particular, the point $1$), there are only two tangents to $E$ through that point.  This implies $H$ is the unique hexagon that includes $1$ and has the required tangency properties.  Thus, our restricted map uniquely determines the numerical range of the matrix and hence uniquely determines the matrix itself.

Suppose that the line through $1$ and $f_5$ also intersects $\bbT$ at $e^{it}$.  We have shown that the space of all $5\times 5$ CMV matrices with the desired property is naturally parametrized by $\{(x,y): 0<x<y<t\}$, which is an open triangle.
\end{proof}

  % % % % % % % % % % % % % % % % % % % % % % % % % % % % % % % % % % % % % % % % % % % % % % % % % % % % % % % % % % % % % % 
  \section{The pentagon case} \label{sec:pentagon}
  % % % % % % % % % % % % % % % % % % % % % % % % % % % % % % % % % % % % % % % % % % % % % % % % % % % % % % % % % % % % % % 

To find Poncelet $5$-ellipses, it will not be possible to consider compositions of Blaschke products as in the previous sections, which partially explains why this case has been studied less often in the literature.  Instead, we will revisit the Mirman system and prove a structure theorem about the set of possible solutions.  In this setting, the system of equations has multiple solutions and our next result describes their relative placement in the plane.
  
  \begin{theorem} \label{thm:pentagon}
  	Let $f_1,f_2\in \mathbb D$, $f_1 f_2\neq 0$ and set $\Phi_2(z)=(z-f_1)(z-f_2)$. The system
  	\begin{equation} \label{caseN=5-3}
  	\begin{split}
  	\frac{\Phi_2(z)}{\Phi_2^*(z)}  &= w f_{1}  , \\
  	\frac{\Phi_2(w)}{\Phi_2^*(w)}&=z f_2
  	\end{split}
  	\end{equation}
  	has exactly 5 distinct solution pairs  $(z,w)\in \bbC^2$:  four of them are in $\mathbb D^2$: $(0,f_2)$, $(f_1,0)$, $(z_1,w_1)$, $(z_2,w_2)$, and exactly one  solution $(z_3,w_3)$ satisfies $|z_3|>1$, $|w_3|>1$.
  	
  	Moreover, the points $z_1$, $z_2$ and $z_3$ are collinear, as are $w_1$, $w_2$ and $w_3$. More precisely,
  	\begin{enumerate}[i)]
  		\item It holds that
		$$
  		\frac{z_1-f_2}{w_1-f_1}=\frac{z_2-f_2}{w_2-f_1}=\frac{z_3-f_2}{w_3-f_1}=\frac{f_1\Phi_2^*(f_2)}{f_2\Phi_2^*(f_1)},
  		$$
  		\item The points $f_1$, $w_1$, $w_2$ and $w_3$ are collinear, i.e.
  		$$
  		\frac{w_i-f_1}{w_j-f_1}\in\bbR, \quad i, j \in \{1,2,3\}, \quad i\neq j. 
  		$$
  		and the points $f_2$, $z_1$, $z_2$ and $z_3$ are collinear, i.e.
  		$$
  		\frac{z_i-f_2}{z_j-f_2}\in\bbR, \quad i, j \in \{1,2,3\}, \quad i\neq j. 
  		$$
  		
  	\end{enumerate}
  \end{theorem}
\begin{figure}[htb]
		\label{fig:MirmanPentagons2}
	\begin{center}
		\begin{tikzpicture}[scale=2,cap=round,>=latex]
			% Preamble
			\tikzset{mypoints/.style={fill=white,draw=black,thick}}
			\def\ptsize{0.7pt}

			% draw the unit circle
			\draw[thick] (0cm,0cm) circle(1cm);
			
			% Define foci f1 and f2
			\def\fx{0.13} \def \fy{0.44} 
			\def\ffx{-0.43} \def \ffy{0} 
			\coordinate (F2) at (\ffx cm,\ffy cm);
			\coordinate (F1) at (\fx cm,\fy cm);
			\draw (F1) node[above=1pt] { $f_1$};
			
			\draw (F2) node[left=1pt] { $f_2$};
			
			% Define the second pair of points and join
			\def\wx{-0.42} \def \wy{0.34} 
			\def\wwx{0.37} \def \wwy{0.49} 
			\coordinate (W1) at (\wx cm,\wy cm);
			\coordinate (W2) at (\wwx cm, \wwy cm);
			\draw[gray,thick] (W1) -- (W2);
			%	\filldraw[black] (85:1cm) circle(0.4pt);
			%	\filldraw[black] (235:1cm) circle(0.4pt);
			\draw (W1) node[above=1pt] { $w_1$};
			\draw (W2) node[right=1pt] { $w_2$};
			
			% Define the third pair of points and join
			\def\zx{-0.14} \def \zy{0.57} 
			\def\zzx{-0.55} \def \zzy{-0.24} 
			\coordinate (Z1) at (\zx cm, \zy cm);
			\coordinate (Z2) at (\zzx cm,\zzy cm);
			\draw[gray,thick] (Z2) -- (Z1);
			%	\filldraw[black] (175:1cm) circle(0.4pt);
			%	\filldraw[black] (50.5:1cm) circle(0.4pt);
			\draw (Z1) node[left=1pt,above=2pt] { $z_1$};
			\draw (Z2) node[left=1pt] { $z_2$};
			
			% Define the fourth pair of points and join
			\coordinate (W3) at (-2.01cm, 0.06cm); 
			\coordinate (W3a) at (-1.22cm, 0.20cm); 
			\coordinate (W3b) at (-1.62cm, 0.12cm); 
			
			\draw[gray,thick] (W1) -- (W3a);
			\draw[gray,thick,dashed] (W3a) -- (W3b);
			\draw[gray,thick] (W3b) -- (W3);

		\coordinate (Z3) at (0.68cm, 2.19cm);
		\coordinate (Z3a) at (0.27cm, 1.38cm);
			\coordinate (Z3b) at (0.48cm, 1.78cm);

		\draw[gray,thick] (Z1) -- (Z3a);
			\draw[gray,thick,dashed] (Z3a) -- (Z3b);
		\draw[gray,thick] (Z3b) -- (Z3);

			\draw (Z3) node[right=5pt,above=1pt] { $z_3$};
			\draw (W3) node[left=1pt] { $w_3$};

			\draw[gray,thick,dashed,-] (F1) -- (F2);
			\draw[gray,thick,dashed,-] (Z1) -- (W1);
			\draw[gray,thick,dashed,-] (Z2) -- (W2);
	
			% Draw the small circles
			\foreach \p in {W1,W2,Z2,Z1,Z3,W3}
			\fill[mypoints] (\p) circle (\ptsize);
			
				\filldraw[black] (F1) circle(\ptsize);
			\filldraw[black] (F2) circle(\ptsize);
			
		\end{tikzpicture}
		\caption{Collinearity of the points $f_1$, $f_2$, $z_j$'s and $w_j$'s, as explained in Theorem~\ref{thm:pentagon}.}
	\end{center}
\end{figure}
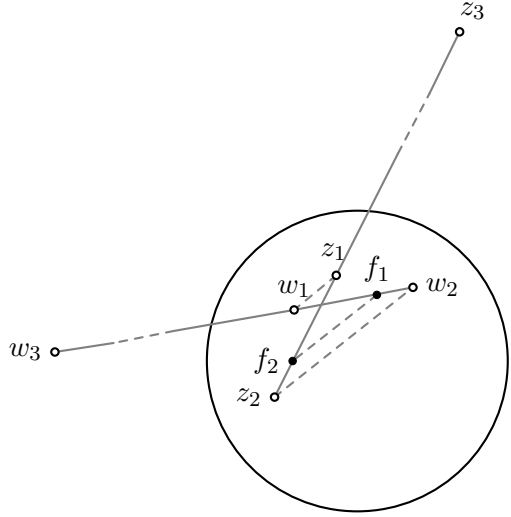

\begin{proof}
Let us denote
  	$$
  	g_1(z)=\frac{1}{f_1}\frac{(z-f_1)(z-f_2)}{(1-z\bar{f}_1)(1-z\bar{f}_2)}, \quad g_2(z)=\frac{1}{f_2}\frac{(z-f_1)(z-f_2)}{(1-z\bar{f}_1)(1-z\bar{f}_2)}, 
  	$$
  	so that \eqref{caseN=5-3} takes the form $g_1(z)=w$, $g_2(w)=z$. From here, we get that \eqref{caseN=5-3} implies that
  	\begin{equation} \label{caseN=5-4}
  	\begin{split}
  	g_2\circ g_1(z)  &= z  , \\
  	g_1\circ g_2(w)  &= w.
  	\end{split}
  	\end{equation}
  	Both $g_2\circ g_1$ and $g_1\circ g_2$ have 4 fixed points inside $\mathbb D$. Indeed, consider  $g_2\circ g_1$ (same analysis for  $g_1\circ g_2$).
  	Recall that a Blaschke product is analytic in $\mathbb D$, maps $\bbT\to \bbT$, $\bbD \to \bbD$ and exterior of $\bbD$ onto its exterior. So,
  	$$
  	|z|=1 \quad \Rightarrow \quad |g_1(z)|>1 \quad \Rightarrow \quad |g_2(g_1(z))|>1 .
  	$$
  	By Rouche's theorem, $g_2\circ g_1(z)-z$ and $g_2\circ g_1(z)$ have the same number of zeros in $\mathbb D$; it is straightforward to check that $g_2\circ g_1$ vanishes at 4 points in $\mathbb D$. 
  	
  	Finally, we claim that the statement of the proposition is equivalent to the just established fact that both $g_2\circ g_1$ and $g_1\circ g_2$ have 4 fixed points inside $\mathbb D$.
  	
  	Indeed, the 4 pairs of solutions of \eqref{caseN=5-3} satisfy \eqref{caseN=5-4}. Reciprocally, let $z$ be a fixed point of $g_2\circ g_1$ and denote $\tau = g_1(z)$. Then $g_2(\tau)=g_2(g_1(z))=z$, and in consequence, $g_1(g_2(\tau))=g_1(z)=\tau$, meaning that $\tau = g_1(z)$ is a fixed point of $g_1\circ g_2$. This shows that the fixed points $z$ and $w$ of $g_2\circ g_1$ and $g_1\circ g_2$ can be paired $(z,w)$ in such a way that \eqref{caseN=5-3} holds.
  	
  	Now we prove the statement about the remaining solution, this time in $(\bbC\setminus \overline{\mathbb D})^2$. The identity
  	\begin{equation} \label{conjug}
  	\overline{\frac{(1/\bar{z}-f_1)(1/\bar{z}-f_2)}{(1-\bar{f}_1/\bar{z})(1-\bar{f}_2/\bar{z})}}=\frac{1}{\frac{(z-f_1)(z-f_2)}{(1-z\bar{f}_1)(1-z\bar{f}_2)}}
  	\end{equation}
  	allows us to reduce the analysis of \eqref{caseN=5-3} in $(\bbC\setminus \overline{\mathbb D})^2$ to the equivalent system
  	\begin{equation} \label{caseN=5outside}
  	\begin{split}
  	\frac{(z-f_1)(z-f_2)}{(1-z\bar{f}_1)(1-z\bar{f}_2)}  &= \frac{w}{\overline{f_{1}}}   , \\
  	\frac{(w-f_1)(w-f_2)}{(1-w\bar{f}_1)(1-w\bar{f}_2)}&=\frac{z}{\overline{f_{2}}} 
  	\end{split}
  	\end{equation}
  	for $(z,w)\in \mathbb D^2$ (we return to the actual solutions outside by the mapping $z \mapsto 1/\overline z$, $w \mapsto 1/\overline w$). The advantage of working in $\mathbb D$ is that again we can use the fixed point argument and Rouche's Theorem. Indeed, as before, define 
  	$$
  	h_1(z)=\overline{f_1} \, \frac{(z-f_1)(z-f_2)}{(1-z\bar{f}_1)(1-z\bar{f}_2)}, \qquad\qquad h_2(z)=\overline{f_2} \, \frac{(z-f_1)(z-f_2)}{(1-z\bar{f}_1)(1-z\bar{f}_2)}, 
  	$$
  	and look for fixed points of $h_1\circ h_2$ and $h_2 \circ h_1$. This time
  	$$
  	|z|=1 \quad \Rightarrow \quad |h_1(z)|<1 \quad \Rightarrow \quad |h_2(h_1(z))|<1 ,
  	$$
  	and by Rouche's Theorem, $h_2\circ h_1(z)-z$ and $f(z)=z$ have the same number of zeros in $\mathbb D$, that is, exactly one.

To prove the statements about collinearity, define the linear functions
\begin{equation}\label{zwfunctions}
\begin{split}
z(t)&=\frac{f_2-f_1}{1-f_1\bar{f}_2}+\frac{4f_1}{(1-|f_1|^2)(1-f_1\bar{f}_2)}t\\
w(t)&=\frac{f_1-f_2}{1-f_2\bar{f}_1}+\frac{4f_2}{(1-|f_2|^2)(1-f_2\bar{f}_1)}t
\end{split}
\end{equation}
It is a straightforward calculation to verify that the two polynomials (in the variable $t$)
\[
\Phi_2(z(t))-w(t)f_1\Phi_2^*(z(t)),\qquad\qquad\Phi_2(w(t))-z(t)f_2\Phi_2^*(w(t))
\]
are scalar multiples of each other so any zero of one is a zero of the other.  Notice that both of these polynomials have degree $3$.  One can also check by hand that if $t_z$ is such that $z(t_z)=0$, then $w(t_z)\neq f_2$.  Similarly, if $t_w$ is such that $w(t_w)=0$, then $z(t_w)\neq f_1$.  Thus, the three zeros $\{t_j\}_{j=1}^3$ of
\begin{equation}\label{phisame}
Y(t):=\Phi_2(z(t))-w(t)f_1\Phi_2^*(z(t))
\end{equation}
will be such that $(z(t_j),w(t_j))$ is a solution to the system \eqref{caseN=5-3} other than $(0,f_2)$ and $(f_1,0)$.  Thus we can say $(z_j,w_j)=(z(t_j),w(t_j))$.  If we set $t_0=\frac{1}{4}(1-|f_1|^2)(1-|f_2|^2)$ and observe that $z(t_0)=f_2$ and $w(t_0)=f_1$, then a short calculation shows
\[
\frac{z_j-f_2}{w_j-f_1}=\frac{z(t_j)-z(t_0)}{w(t_j)-w(t_0)}=\frac{f_1\Phi_2^*(f_2)}{f_2\Phi_2^*(f_1)}.
\]
This proves claim (i) of the theorem.

To prove claim (ii), it suffices to show that each $t_j\in\bbR$.  To this end, a calculation reveals that if we divide the polynomial $Y(t)$ in \eqref{phisame} by its leading coefficient, then we obtain a monic degree $3$ polynomial with real coefficients.

Define
\[
q(x,y;t):=\left(y+\frac{x-f_1}{1-\overline{f_1}x}\right)\left(y+\frac{x-f_2}{1-\overline{f_2}x}\right)-\frac{4txy}{(1-\overline{f_1}x)(1-\overline{f_2}x)}
\]
There exist two distinguished matrices $A_1,A_2 \in S_4$ such that $A_1$ has numerical range bounded by an ellipse with foci at $\{f_1,f_2\}$ and the pentagram curve of $A_2$ is an ellipse with foci at $\{f_1,f_2\}$.  Let us denote the eigenvalues of $A_j$ by $\{f_1,f_2,z_j,w_j\}$ for $j=1,2$.
Then from \cite[Section 5]{HMFPS} (and also \cite[Equations 29--32]{Mirman1}) we know that there exist positive real numbers $b_1$ and $b_2$ so that
\[
q(f_2,z_j;b_j^2)=q(w_j,f_1;b_j^2)=q(z_j,w_j;b_j^2)=0
\]
for $j=1,2$.  In fact, $b_j$ is the length of the minor semiaxis of the ellipse associated with $A_j$ whose foci are $\{f_1,f_2\}$ (see \cite{Mirman1}).

Using $q(f_2,z_j;b_j^2)=q(w_j,f_1;b_j^2)=0$ and the formulas \eqref{zwfunctions}, we find that $z_j=z(b_j^2)$ and $w_j=w(b_j^2)$.  A short calculation shows (recall $t_0=(1-|f_1|^2)(1-|f_2|^2)/4$)
\[
q(z(t),w(t);t)=\frac{Y(t)(t-t_0)}{C_{f_1,f_2} \Phi_2^*(z(t))},
\]
where $\displaystyle C_{f_1,f_2}= \frac{-f_1 t_0}{f_2}$.  Thus the zeros of $Y(t)$ are also zeros of $q(z(t),w(t);t)$.

If $t_0$ were a zero of $Y(t)$, then $z=z(t_0)$ and $w=w(t_0)$ would satisfy \eqref{caseN=5-3}.  However, we have seen that $z(t_0)=f_2$ and $w(t_0)=f_1$, giving $0=\Phi_2(f_2)=f_1^2$, which is nonzero by assumption, so $Y(t_0)\neq0$.

Thus the zeros of $Y(t)$ are zeros of $q(z(t),w(t);t)$ distinct from $t_0$.  We know that $q(z(t),w(t);t)$ has zeros at $t=b_1^2$  and $t=b_2^2$ so these must be zeros of $Y(t)$ as well.  This gives us two real zeros of $Y(t)$.  Our earlier observation implies that $Y(t)$ has either one or three real zeros, so all zeros of $Y(t)$ are real as desired.
\end{proof}

\begin{figure}[htb] 
	\begin{center}
		\begin{tabular}{c p{10mm}c}		
			\begin{overpic}[width=0.35\linewidth]{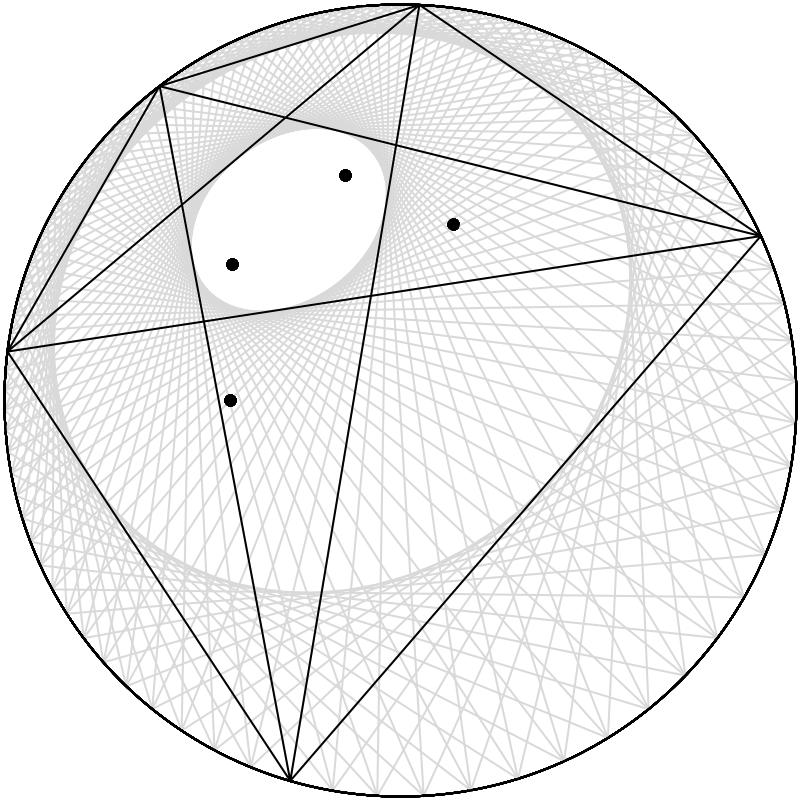}
				\put(49,110){ $w_1$}
				\put(58,125){ $z_1$}
				\put(47,80){ $f_2$}
				\put(83,110){ $f_1$}
			\end{overpic}
			& \ 	&
			\begin{overpic}[width=0.35\linewidth]{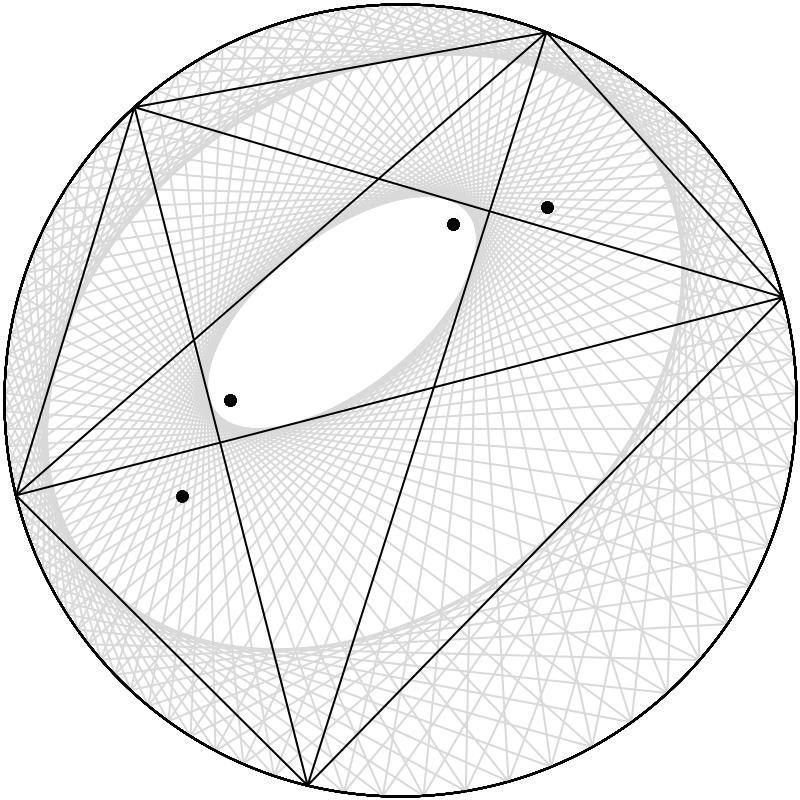}
				\put(110,125){ $w_2$}
				\put(36,55){ $z_2$}
				\put(47,80){ $f_2$}
				\put(83,110){ $f_1$}
			\end{overpic}
			
		\end{tabular}
	\end{center}
	\caption{Pairs $(z_1, w_1) $ and $(z_2, w_2)$ as foci of the pentagram and the Poncelet ellipses, respectively.}
	\label{fig:MirmanPentagons}
\end{figure}
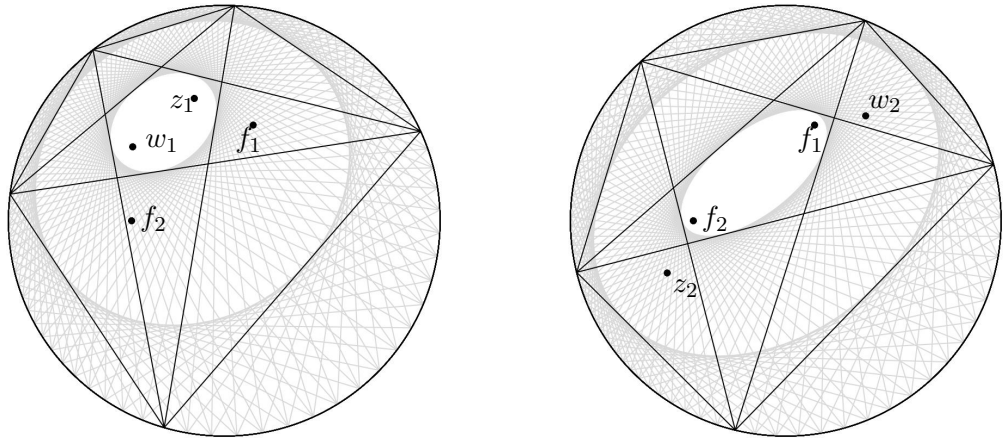

The solutions to the Mirman system in this setting have a geometric interpretation.  Given $\{f_1,f_2\}\in\bbD^2$, we have just seen that there are three solution pairs $(z,w)$ to the Mirman system and we denote them by $(z_1,w_1)$, $(z_2,w_2)$, and $(z_3,w_3)$.  One of the solution pairs in $\bbD^2$ (say $(z_1,w_1)$) will be such that the $4\times 4$ cutoff CMV matrix with eigenvalues $\{f_1,f_2,z_1,w_1\}$ will have numerical range bounded by an ellipse with foci at $f_1$ and $f_2$ and the pentagram curve of this matrix will be an ellipse with foci at $z_1$ and $w_1$ (see Figure~\ref{fig:MirmanPentagons}, left).  For the other solution pair in $\bbD^2$ (say $(z_2,w_2)$), it will be true that the $4\times 4$ cutoff CMV matrix with eigenvalues $\{f_1,f_2,z_2,w_2\}$ has numerical range bounded by an ellipse with foci at $z_2$ and $w_2$ and the pentagram curve of this matrix will be an ellipse with foci at $f_1$ and $f_2$, see Figure~\ref{fig:MirmanPentagons}, right.  The geometric interpretation of the solution $(z_3,w_3)\in(\bbC\setminus\bar{\bbD})^2$ is not clear.

Notice that Theorem \ref{thm:pentagon} implies that each line $\ell_j$ passing through $\{z_j,w_j\}$ for $j=1,2,3$ is parallel to the line through $\{f_1,f_2\}$.  For $j=1,2$, one can obtain this same conclusion from the fact that the ellipses $C_1$ and $C_2$ corresponding to a matrix $A\in S_4$ are in the same package (see \cite{MS}).  The fact that this same conclusion applies to $\ell_3$ is a new result.

% % % % % % % % % % % % % % % % % % % % % % % % % % % % % % % % % % % % % % % % % % % % % % % % % % % % % % % % % 

\section*{Acknowledgments}

The second author was partially supported by Simons Foundation Collaboration Grants for Mathematicians (grant 710499) and by the Spanish Government--European Regional Development Fund (grant MTM2017-89941-P), Junta de Andaluc\'{\i}a (research group FQM-229 and Instituto Interuniversitario Carlos I de F\'{\i}sica Te\'orica y Computacional), and by the University of Almer\'{\i}a (Campus de Excelencia Internacional del Mar  CEIMAR).

The fourth author graciously acknowledges support from Simons Foundation Collaboration Grant 707882.

 % % % % % % % % % % % % % % % % % % % % % % % % % % % % % % % % % % % % % % % % % % % % % % % % % % % % % % % % 
 
%

\end{document}